\documentclass{amsart}
\usepackage{amssymb}
\usepackage{verbatim}
\usepackage[usenames]{color}
\usepackage{hyperref}
\usepackage{url}
\usepackage[all]{xy}

\newtheorem{theorem}{Theorem}[section]
\newtheorem{proposition}[theorem]{Proposition}
\newtheorem{lemma}[theorem]{Lemma}
\newtheorem{cor}[theorem]{Corollary}

\newtheorem{observation}[theorem]{Observation}
\newtheorem*{thm1.3}{Theorem \ref{thm.barren}}
\newtheorem*{thm1.4}{Theorem \ref{thm.spcpres}}
\newtheorem*{thm1.5}{Theorem \ref{thm.maintRs}}

\theoremstyle{definition}
\newtheorem{definition}[theorem]{Definition}

\newtheorem{remark}[theorem]{Remark}
\newtheorem{example}[theorem]{Example}
\newtheorem{question}[theorem]{Question}

\theoremstyle{remark}

\newcommand{\defemph}{\textit}

\newcommand{\dom}{\textrm{Dom}}

\newcommand{\zfc}{\textrm{ZFC}}

\newcommand{\ad}{\textrm{AD}}
\newcommand{\im}{\textrm{Im}}

\newcommand{\mc}{\mathcal}
\newcommand{\mbb}{\mathbb}

\newcommand{\forces}{\Vdash}

\newcommand{\baire}{{^\omega}\omega}

\newcommand{\Fin}{\mathrm{Fin}}

\newcommand{\al}{\alpha}
\newcommand{\om}{\omega}

\newcommand{\sse}{\subseteq}

\DeclareMathOperator{\depth}{depth}

\DeclareMathOperator{\FIN}{FIN}

\newcommand{\re}{\restriction}

\newcommand{\bN}{\mathbb{N}}

\newcommand{\ra}{\rightarrow}

\newcommand{\lgl}{\langle}
\newcommand{\rgl}{\rangle}

\newcommand{\Fraisse}{Fra{\"{i}}ss{\'{e}}}

\newcommand{\Nesetril}{Ne{\v{s}}et{\v{r}}il}

\newcommand{\Rodl}{R{\"{o}}dl}

\theoremstyle{remark}

\newcommand{\noprint}[1]{\relax}

\begin{document}

\title{Classes of barren extensions}

\author{Natasha
  Dobrinen}
\address{Department of Mathematics\\
  University of Denver \\
C.M. Knudson Hall, Room 300
2390 S. York St.
   \\ Denver, CO \ 80208 U.S.A.}
\email{Natasha.Dobrinen@du.edu}
\urladdr{\url{http://web.cs.du.edu/~ndobrine}}
\thanks{The first author was partially supported by  National Science Foundation Grants  DMS-1301665 and DMS-1901753}

\author{Dan
 Hathaway}
\address{
Department of Mathematics \\
University of Vermont\\
Innovation Hall \\
82 University Place\\
Burlington, VT 05405 U.S.A.}
\email{Daniel.Hathaway@uvm.edu}
\urladdr{\url{http://mysite.du.edu/~dhathaw2/}}

%


\maketitle

\begin{abstract}
Henle, Mathias, and Woodin proved in \cite{Henle/Mathias/Woodin85} that,  provided that  $\om\ra(\om)^{\om}$ holds in a model $M$ of ZF, then 
 forcing with $([\om]^{\om},\sse^*)$ over $M$ adds no new sets of ordinals, thus earning the name a ``barren'' extension.
 Moreover, under an  additional assumption,  they proved that this generic extension preserves all strong partition cardinals.
This forcing thus produces a model $M[\mathcal{U}]$,  where
$\mathcal{U}$  is  a Ramsey ultrafilter,   with many properties of the original model  $M$.
 This begged the question of how important the Ramseyness of $\mathcal{U}$ is for these results. 
 In this paper, we show that several classes of 
 $\sigma$-closed  forcings which  generate non-Ramsey ultrafilters have the same properties.
Such ultrafilters include Milliken-Taylor ultrafilters, a class of rapid p-points of Laflamme,   $k$-arrow p-points of Baumgartner and Taylor, and extensions to a  class of ultrafilters  constructed by Dobrinen, Mijares and Trujillo.
 Furthermore, the class of  Boolean algebras $\mathcal{P}(\om^{\al})/\Fin^{\otimes \al}$, $2\le \al<\om_1$,  forcing  non-p-points  also produce barren extensions.
\end{abstract}


\section{Introduction}\label{sec.intro}

In their paper, {\em A barren extension} \cite{Henle/Mathias/Woodin85},  Henle, Mathias, and Woodin proved that forcing with $([\om]^{\om},\sse^*)$   
does not add new subsets of ordinals and preserves strong partition cardinals, assuming the ground model satisfies certain  properties.
The first of these is the infinite partition relation 
\begin{equation}
\om\ra(\om)^{\om},
\end{equation}
which means that for each coloring $c:[\om]^{\om}\ra 2$, there is an infinite subset $x\sse \om$ such that the restriction of  $c$ to $[x]^{\om}$ is constant.
This partition relation fails outright in the presence of the Axiom of Choice.  
However, it is consistent   with fragments of Choice, as seen in the following:
AD$_{\mathbb{R}}$   implies $\om\ra(\om)^{\om}$.
This was first proved by Prikry with the additional assumption of DC in \cite{Prikry76}, and  soon after, Mathias  showed DC  was unnecessary in  \cite{Mathias77}.
Similarly 
AD$^+$ + $V=L(\mathcal{P}(\mathbb{R}))$ implies 
 $\om\ra(\om)^{\om}$.
 This was proved in the Cabal and can be seen to 
 follow from $\Sigma^2_1$ reflection to the Suslin coSuslin sets  (see \cite{Steel/Trang}
 and Theorem 25 in \cite{Woodin10})
 and from the fact that every Suslin set of reals is Ramsey
 (see Theorem 2.2 of 
 \cite{Feng/Magidor/Woodin92}).
Hence AD + $V = L(\mathbb{R})$ implies $\om\ra(\om)^{\om}$
 (because AD + $V = L(\mathbb{R})$ implies AD$^+$ + $V=L(\mathcal{P}(\mathbb{R}))$),
 and so $\om\ra(\om)^{\om}$ holds in $L(\mathbb{R})$
 assuming there are $\omega$ Woodin cardinals with a measurable above
 (see \cite{Shelah/Woodin90}).
Furthermore, the partition relation $\om\ra(\om)^{\om}$  holds  in  the $L(\mathbb{R})$ of a model $V$ of ZFC 
if  
 $V$ is a model 
 obtained by collapsing a strongly inaccessible cardinal $\kappa$ to $\om_1$ via the L\'{e}vy collapse (see \cite{Mathias77}).

Henle, Mathias, and Woodin proved that forcing with  $([\om]^{\om},\sse^*)$ over a model satisfying the infinitary partition relation $\om\ra(\om)^{\om}$ does not add any new subsets of ordinals over the ground model, aptly  calling this extension  {\em barren}.

\begin{theorem}[Henle, Mathias, and Woodin, \cite{Henle/Mathias/Woodin85}]\label{thm.1.0}
Let $M$ be a transitive model of ZF + $\om\ra(\om)^{\om}$ and $N$ be its extension via $([\om]^{\om},\sse^*)$.
Then $M$ and $N$ have the same sets of ordinals;
moreover every wellordered sequence in $N$ of elements of $M$ lies in $M$.
\end{theorem}

The other key  properties  utilized in \cite{Henle/Mathias/Woodin85} are  called  LU and EP.
LU is a uniformization axiom, guaranteeing   a function  uniformizing  a relation on the Baire space, relative to some infinite set.
EP says that 
the intersection of any well-ordered collection of completely Ramsey positive sets is again completely Ramsey positive.
These assumptions are shown to preserve infinite dimensional partition relations of the following type.
Given cardinals $\kappa,\lambda,\mu$ with
$2\le\mu<\kappa$ and 
$\lambda\le \kappa$,
\begin{equation}\label{eq.spc}
\kappa\ra(\kappa)^{\lambda}_{\mu}
\end{equation}
denotes that for each coloring $c:[\kappa]^{\lambda}\ra\mu$,
there is a subset $K\in [\kappa]^{\kappa}$ such that the restriction of $c$ to $[K]^{\lambda}$ is constant.
An uncountable cardinal $\kappa$ satisfying $(\ref{eq.spc})$
 for $\lambda = \kappa$ and for every $2 \le \mu < \kappa$ is called a {\em strong partition cardinal}.

\begin{theorem}[Henle, Mathias, and Woodin, \cite{Henle/Mathias/Woodin85}]\label{thm.3.3}
(ZF + EP + LU)\ \ Suppose $0<\lambda=\om\cdot\lambda\le \kappa$, $2\le\mu<\kappa$, 
$\kappa\ra(\kappa)^{\lambda}_{\mu}$,
and that there is a surjection from $[\om]^{\om}$ onto $[\kappa]^{\kappa}$.
Then in the forcing extension via $([\om]^{\om},\sse^*)$, $\kappa\ra(\kappa)^{\lambda}_{\mu}$ holds.
\end{theorem}

Forcing with 
 $([\om]^{\om},\sse^*)$  adds an ultrafilter  $\mathcal{U}$ which is {\em Ramsey}, meaning that given any  $l,n\ge 2$,
   $X\in\mathcal{U}$, and coloring $c:[X]^n\ra l$, there is some $U\in\mathcal{U}$ with $U\sse X$ such that the restriction of $c$ to $[U]^n$ is constant.
 This is written as
 \begin{equation}
 \mathcal{U}\ra(\mathcal{U})^n_l.
 \end{equation}
 It is shown  in Proposition 4.1 of \cite{Henle/Mathias/Woodin85} that 
 the hypotheses,  EP + LU,  of Theorem \ref{thm.3.3}  hold in $V$ if it satisfies   AD + $V=L(\mathbb{R})$.
 In this case, $V[\mathcal{U}]$   preserves the strong partition cardinals in  $V$ mentioned in that theorem.

One cannot help but wonder, how important is the Ramseyness of the generic ultrafilter $\mathcal{U}$ for  these results?
Are there  forcings which add non-Ramsey ultrafilters for which the consequences of 
Theorems \ref{thm.1.0} and \ref{thm.3.3}
still hold?
The  main tools of the proofs,  $\om\ra(\om)^{\om}$, EP and LU, 
 involve   properties   of $[\om]^{\om}$  as a topological space.
Thus, we surmised that other forcings  with similar  topological  properties 
would likely add ultrafilters with barren extensions.
This turned out to be the case.
In this paper, we prove that several  collections  of 
 $\sigma$-closed  partial orders  forcing  non-Ramsey 
 ultrafilters  produce 
the same conclusions as Theorems \ref{thm.1.0} and \ref{thm.3.3}.

 The natural place to look for forcings satisfying analogues of $\om\ra(\om)^{\om}$  is
  topological Ramsey spaces, as such spaces, by definition,   satisfy analogues of this infinitary partition relation for definable sets. 
  These spaces are defined in Section \ref{sec.2}, which provides   basic background  and  their connection with forced ultrafilters.
Topological Ramsey  spaces have been shown to form  dense subsets 
of many  forcings which add ultrafilters satisfying weak partition relations.  
This  includes  constructions  in 
  \cite{DobrinenJSL16}, \cite{DobrinenJML16}, 
   \cite{Dobrinen/Mijares/Trujillo17}, \cite{Dobrinen/Todorcevic14}, and  \cite{Dobrinen/Todorcevic15}, which were motivated by questions on  exact  
 Rudin-Keisler and Tukey structures  below such ultrafilters.
 An exposition of this area is found in \cite{DobrinenSEALS19}.
In this paper,  we utilize  topological Ramsey space techniques  to extend  results of Henle, Mathias, and Woodin. We first state the results for general forcing posets.

The following  two theorems  extend Theorems \ref{thm.1.0}
and \ref{thm.3.3}.
For the notions of 
extended coarsened poset and the Left-Right Axiom 
assumed in the next theorem, see
Definitions \ref{defn.sscp} and \ref{LRA}. 
We say that {\em all  cubes  of a poset $\lgl X,\le\rgl$ are Ramsey} if the following holds:
 Given 
 $x \in X$, a positive integer $k$,  and  a coloring
 $c: \{y\in X:y\le x\}\to k$,
 there is some $y\le x$ 
 such that $c \restriction \{z\in X:z\le y\}$ is constant.

 \begin{theorem}\label{thm.barren}
 Let
$M$ be a transitive  model of ZF. 
In $M$, let $\lgl X,\le\rgl$ be a forcing poset 
and assume that    $\le^*$ is a $\sigma$-closed coarsening
 of $\le$ such that $\lgl X,\le\rgl$ and $\lgl X,\le^*\rgl$ have isomorphic separative quotients. 
Suppose that this coarsening (or a forcing equivalent one) satisfies   the Left-Right Axiom, and suppose
all cubes  of $\lgl X,\le\rgl$ are Ramsey.
 Let $N$ be a 
generic extension of $M$ by the forcing
 $\langle X, \le \rangle$.
 Then $M$ and $N$ have the same sets of ordinals;
 moreover, every sequence in $N$ of elements of $M$ lies in $M$.
\end{theorem}

In Lemma \ref{lem.ADimpliesRamsey}, 
we prove that if either AD$_{\mathbb{R}}$ or AD$^+$ + $V=\mathcal{P}(\mathbb{R})$ hold, then 
 all subsets of a topological Ramsey space are Ramsey, which makes such spaces a source of natural examples producing barren extensions.

 In Section \ref{sec.4}, we prove the extension of Theorem \ref{thm.3.3}.
 In the following, 
 LU$(\mbb{P})$ and EP$(\mbb{P})$ are generalizations of the axioms LU and EP of Henle, Mathias, and Woodin (see Definitions \ref{def.LUP} and \ref{def.EPP}).
 A subset $S$ of a partial ordering $\lgl X,\le\rgl$ is {\em Ramsey} if for each $p\in X$, there exists a $q\le p$ such that $\{r\in X:r\le q\}$ is either contained in or disjoint from $S$.

\begin{theorem}\label{thm.spcpres}
Suppose $\kappa \rightarrow (\kappa)^\lambda_\mu$,
 where $\kappa, \lambda, \mu$ are non-zero ordinals such that
 $\lambda = \omega \lambda \le \kappa$ and
 $2 \le \mu < \kappa$.
Suppose also that there is a surjection
 from ${^\omega 2}$ to $[\kappa]^\kappa$.
Let $\mbb{P} = \langle X, \le, \le^* \rangle$
 be a coarsened poset such that
 $\mbox{EP}(\mbb{P})$ and each $=^*$-equivalence class is countable.
Assume every $S \subseteq X$ is Ramsey.
If $\mbox{LU}(\mbb{P})$ holds and $\langle X, \le \rangle$  adds
no new sets of ordinals,
 then $\langle X, \le \rangle$ forces
 $\kappa \rightarrow (\kappa)^\lambda_\mu$.
\end{theorem}

It follows from these theorems that topological Ramsey spaces with natural $\sigma$-closed coarsenings 
force 
  barren extensions containing ultrafilters and  preserving the strong partition cardinals in the ground model.
By an {\em axiomatized} topological Ramsey space, we mean  one that satisfies the four axioms of Todorcevic in \cite{TodorcevicBK10} and is closed as a subspace of an appropriate product space (see Definition \ref{def.axiomatizedtRs}).
All known examples of topological Ramsey spaces are axiomatized. 
We will  say that a forcing $\mathbb{P}$ is 
 {\em Ramsey-like}
 if it is forcing equivalent to some axiomatized topological Ramsey space with an ($\sigma$-closed) extended  coarsening  satisfying the Left-Right Axiom.
 The next  theorem follows from the previous two.

\begin{theorem}\label{thm.maintRs}
Assume $M$ satisfies ZF +  either 1) AD$_{\mathbb{R}}$ or    2) AD$^+ + V=L(\mathcal{P}(\mathbb{R}))$.
Let $\mbb{P}$  be a Ramsey-like  forcing, and let 
 $\mathcal{U}$ be  an    ultrafilter  forced  by   $\mbb{P}$.
 Then  $M$ and $M[\mathcal{U}]$ have the same sets of ordinals; moreover, every sequence  in $M[\mathcal{U}]$
 of elements of   $M$ lies in $M$.
 If, further, 
 the  $\sigma$-closed coarsening $\le^*$ has  countable $=^*$-equivalences classes,
 then   
 $\kappa \rightarrow (\kappa)^\lambda_\mu$ holds in $M[\mathcal{U}]$
  whenever
  $\kappa \rightarrow (\kappa)^\lambda_\mu$ holds in  $M$, 
 where $\kappa, \lambda, \mu$ are non-zero ordinals such that
 $\lambda = \omega \lambda \le \kappa$,
 $2 \le \mu < \kappa$, and 
 there is a surjection in $M$
 from ${^\omega 2}$ to $[\kappa]^\kappa$.
 \end{theorem}

 Instances of Theorem \ref{thm.maintRs}
 are seen in 
   Sections \ref{sec.5} and \ref{sec.6},
where  $\mathcal{U}$ ranges over a large collection of non-Ramsey ultrafilters.
 The following examples are indicative of  the types of  ultrafilters 
for which our results guarantee barren extensions.
First, there  are the Milliken-Taylor ultrafilters studied by Mildenberger in \cite{Mildenberger11} which form  a hierarchy 
extending the 
  stable ordered union ultrafilters of Blass in \cite{Blass87}.
It is shown in  Subsection \ref{subsec.FIN} that  these forcings are Ramsey-like.

  In Subsection \ref{subsec.IS} we  present a property called {\em Independent Sequencing} which,  when satisfied, guarantees that a forcing is Ramsey-like.  
  All  of the  forcings  in the remainder of Section \ref{sec.5} have this property, and so  their generic ultrafilters satisfy Theorem \ref{thm.maintRs}.

Our second class of forcings, seen in 
Subsection \ref{subsec.Laflamme},  is
a collection of forcings 
 constructed  by Laflamme  in \cite{Laflamme89}  which   extend the forcing 
$([\om]^{\om},\sse^*)$ by restricting the partial order 
 to produce  a  hierarchy of ultrafilters 
$\mathcal{U}_{\al}$, $\al<\om_1$, for which the partition relations get weaker and weaker.
Laflamme proved that these ultrafilters form a decreasing chain 
under Rudin-Keisler reduction of order type  $(\al+1)^*$,
where the minimum ultrafilter is Ramsey, and the next one above it is weakly Ramsey. 
For each $1\le k<\om$, $\mathcal{U}_k$ satisfies  the following partition relation: 
Given a coloring  $c$ of $[\om]^2$ into finitely many colors, there is a member of $\mathcal{U}_k$ on which $c$ takes no more than $k+1$ colors. 
All of the ultrafilters $\mathcal{U}_{\al}$ have interesting combinatorial properties, but for $\al$ infinite, there are no finite  bounds for colorings of pairs.

A third collection of  Ramsey-like forcings 
 includes   forcings of Baumgartner and Taylor in \cite{Baumgartner/Taylor78} which generate
 $k$-arrow, not $(k+1)$-arrow ultrafilters, as well as a forcing of Blass in \cite{Blass73} producing a p-point with two Rudin-Keisler-incomparable predecessors.
In Subsection \ref{subsec.DMT}, we present these ultrafilters as well as a  general class of
forcings 
 constructed in \cite{Dobrinen/Mijares/Trujillo17}  which encompass these as special cases.
These forcings are shown to be Ramsey-like, and hence, their forced  ultrafilters satisfy Theorem \ref{thm.maintRs}.

In Section \ref{sec.6}, 
we  investigate another line  of  forcings  of stratified
complexity over
  $([\om]^{\om},\sse^*)$.
Noting that 
  $([\om]^{\om},\sse^*)$ is forcing equivalent to $\mathcal{P}(\om)/\Fin$,
we work with the 
natural hierarchy of Boolean algebras $\mathcal{P}(\om^{\al})/\Fin^{\otimes \al}$, where for $1\le \al<\beta<\om_1$,
the projection of  $\mathcal{P}(\om^{\al})/\Fin^{\otimes \al}$ to  the first $\beta$ coordinates recovers
 $\mathcal{P}(\om^{\beta})/\Fin^{\otimes \beta}$. 
These forcings generate non-p-points for $\al\ge 2$, which still satisfy  weak partition relations.
For example, the generic ultrafilter $\mathcal{G}_2$ forced by 
$\mathcal{P}(\om^{2})/\Fin^{\otimes 2}$
satisfies the following partition property:
Given a coloring $c$ of $[\om^2]^2$ into finitely many colors, 
there is a member of $\mathcal{G}_2$ on which $c$ takes at most four colors. 
Each of these forcings has been shown to contain dense subsets forming topological Ramsey spaces (see \cite{DobrinenJSL16} and \cite{DobrinenJML16}), so  Theorem \ref{thm.barren} holds for these forcings.
However, we do not know if  they preserve strong partition cardinals, since their $=^*$-equivalence classes have cardinality continuum.

\vskip.1in

 \noindent \bf Acknowledgments. \rm
 The authors would like to thank Carlos DiPrisco, Paul Larson, and Adrian Mathias for  their generous discussions which greatly benefited this paper, and the anonymous referee for ways to improve clarity.


\section{Background: Topological Ramsey spaces, infinite partition relations, and associated ultrafilters}\label{sec.2}

A key assumption in the results in  \cite{Henle/Mathias/Woodin85} is the infinitary partition relation $\om\ra(\om)^{\om}$. 
That this holds in models of AD$_{\mathbb{R}}$ or AD$^+$ + $V=L(\mathcal{P}(\mathbb{R})$ is connected with 
a topological characterization of the Ramsey property due to Ellentuck.
Let  $\tau$ be the topology generated by basic open sets of the form
$$
[a,x]=\{y\in[x]^{\om}:a\sqsubset y\},
$$
where $a\in[\om]^{<\om}$ and $x\in[\om]^{\om}$,
and call $[\om]^{\om}$ with  this topology the {\em Ellentuck space}.
Here $a\sqsubset y$ means there is an $n \in \omega$
 such that $a = \{0,1,...,n-1\} \cap y$
 ($a$ is an initial segment of $y$).
Notice that $\tau$ refines the metric topology on the Baire space $[\om]^{\om}$.
Culminating a line of work beginning with Nash-Williams  as to which subsets of the Baire space satisfy an infinite partition relation
(see for instance, \cite{NashWilliams65}, \cite{Galvin/Prikry73} and  \cite{Silver71})
Ellentuck proved the following theorem.

\begin{theorem}[Ellentuck,  \cite{Ellentuck74}]
A subset  $S\sse[\om]^{\om}$  has the property of Baire with respect to the Ellentuck topology if and only if the following holds:
For each non-empty  Ellentuck basic open set $[a,x]$, there is a $y\in [a,x]$  such that 
either $[a,y]\sse S$ or else $[a,y]\cap S=\emptyset$.
\end{theorem}

In particular, for each subset $S\sse[\om]^{\om}$ with the property of Baire in the Ellentuck topology, for any  $x\in [\om]^{\om}$ there is some $y\in[x]^{\om}$ such that  
$[y]^{\om}$ is either contained in $S$ or disjoint from $S$. 
Thus, $\om\ra(\om)^{\om}$ holds models of ZF where all sets of reals are sufficiently definable.

Carlson and Simpson in \cite{Carlson/Simpson90} extracted properties  responsible for infinitary partition relations  on more general spaces, for instance,  spaces of infinite sequences of finite words, and called such spaces topological Ramsey spaces.
Building on their work, Todorcevic presented four axioms in \cite{TodorcevicBK10} which are responsible for similar infinitary partition relations on a wider array of topological spaces. 
The next subsection provides  the minimal  background necessary for understanding this paper.


\subsection{Topological Ramsey spaces}\label{subsec.tRs}

Most of the material in this subsection comes from   Chapter 5  in \cite{TodorcevicBK10}, with a few  new definitions which will help the exposition of this paper.
Axioms {\bf A.1}\rm --{\bf A.4 } \rm  below
are defined for triples
$(\mathcal{R},\le,r)$
of objects with the following properties:
$\mathcal{R}$ is a nonempty set,
$\le$ is a quasi-ordering on $\mathcal{R}$,
 and $r:\mathcal{R}\times\om\ra\mathcal{AR}$ is a  map producing the sequence $(r_n(\cdot)=r(\cdot,n))$ of  restriction maps, where
$\mathcal{AR}$ is  the collection of all finite approximations to members of $\mathcal{R}$.
For $u\in\mathcal{AR}$ and $X\in\mathcal{R}$,
\begin{equation}
[u,X]=\{Y\in\mathcal{R}:Y\le X\mathrm{\ and\ }(\exists n)\ r_n(Y)=u\}.
\end{equation}

\begin{enumerate}
\item[\bf A.1]\rm
\begin{enumerate}
\item[(1)]
$r_0(X)=\emptyset$ for all $X\in\mathcal{R}$.\vskip.05in
\item[(2)]
$X\ne Y$ implies $r_n(X)\ne r_n(Y)$ for some $n$.\vskip.05in
\item[(3)]
$r_m(X)=r_n(Y)$ implies $m=n$ and $r_k(X)=r_k(X)$ for all $k<n$.\vskip.1in
\end{enumerate}
\end{enumerate}

Let $\mathcal{AR}_n=\{r_n(X):X\in\mathcal{R}\}$.
It follows from \bf A.1 \rm(1) that $\mathcal{AR}_0=\{\emptyset\}$.
By  \bf A.1 \rm(3), the sets  $\mathcal{AR}_m$ and $\mathcal{AR}_n$ are disjoint whenever $m\ne n$. 
For each $u \in \mc{AR}_n$ and $m \le n$, let $r_m(u)$ be defined to be
 $r_m(X)$ where $X$ is any element of $\mc{R}$ such that $r_n(X) = u$.
Given $u \in \mc{AR}$,
 let $|u|$ denote the length of $u$.
That is, $|u|$ equals the integer $n$ such that $u \in \mc{AR}_n$.
Said another way, $|u|$ is the integer $n$ such that $r_n(u) = u$.
For $u,v\in\mathcal{AR}$, we write $u\sqsubset v$ exactly when 
$r_{|u|}(v)=u$.
For $n>|u|$, let  $r_n[u,X]$  denote the collection of all $v\in\mathcal{AR}_n$ such that $u\sqsubset v$ and $v\le_{\mathrm{fin}} X$.

\begin{enumerate}
\item[\bf A.2]\rm
There is a quasi-ordering $\le_{\mathrm{fin}}$ on $\mathcal{AR}$ such that\vskip.05in
\begin{enumerate}
\item[(1)]
$\{v\in\mathcal{AR}:v\le_{\mathrm{fin}} u\}$ is finite for all $u\in\mathcal{AR}$,\vskip.05in
\item[(2)]
$Y\le X$ iff $(\forall n)(\exists m)\ r_n(Y)\le_{\mathrm{fin}} r_m(X)$,\vskip.05in
\item[(3)]
$\forall u,v,y\in\mathcal{AR}[y\sqsubset v\wedge v\le_{\mathrm{fin}} u\ra\exists x\sqsubset u\ (y\le_{\mathrm{fin}} x)]$.\vskip.1in
\end{enumerate}
\end{enumerate}

Given an $X\in\mathcal{R}$, we define $\mathcal{AR}\re X$ to be the set of all finite approximations to members $Y\in \mathcal{R}$ such that  $Y\le X$. 
Note that \bf A.2 \rm(1) and (2) imply that for any $X\in\mathcal{R}$, 
\begin{equation}
\mathcal{AR}\re X=\{u\in\mathcal{AR}:(\exists m)\, u\le_{\mathrm{fin}} r_m(X)\}
\end{equation}
and hence, $\mathcal{AR}\re X$ is countable.
This fact will be important throughout the paper.

The number $\depth_X(u)$ is the least $n$, if it exists, such that $u\le_{\mathrm{fin}}r_n(X)$.
If such an $n$ does not exist, then we write $\depth_X(u)=\infty$.
If $\depth_X(u)=n<\infty$, then $[\depth_X(u),X]$ denotes $[r_n(X),X]$.

\begin{enumerate}
\item[\bf A.3] \rm
\begin{enumerate}
\item[(1)]
If $\depth_X(u)<\infty$ then $[u,Y]\ne\emptyset$ for all $Y\in[\depth_X(u),X]$.\vskip.05in
\item[(2)]
$Y\le X$ and $[u,Y]\ne\emptyset$ imply that there is $Y'\in[\depth_X(u),X]$ such that $\emptyset\ne[u,Y']\sse[u,Y]$.\vskip.1in
\end{enumerate}
\end{enumerate}

\begin{enumerate}
\item[\bf A.4]\rm
If $\depth_X(u)<\infty$ and if $\mathcal{O}\sse\mathcal{AR}_{|u|+1}$,
then there is $Y\in[\depth_X(u),X]$ such that
$r_{|u|+1}[u,Y]\sse\mathcal{O}$ or $r_{|u|+1}[u,Y]\sse\mathcal{O}^c$.\vskip.1in
\end{enumerate}

Axiom \bf A.1 \rm implies that 
the map
$\iota:\mathcal{R}\ra \prod_{n<\om}\mathcal{AR}_n$ defined by
\begin{equation}
\iota(X)=\lgl r_n(X):n<\om\rgl,
\end{equation}
for $X\in\mathcal{R}$, 
is an injection.
Note that 
$\prod_{n<\om}\mathcal{AR}_n$ is a subspace of
$\mathcal{AR}^{\mathbb{N}}$, where the set $\mathcal{AR}$ is given the discrete topology; this slightly streamlined notation us commonly used in topological Ramsey space theory. 
Recalling the remark after axiom \bf A.2\rm, 
we may assume without loss of generality that $\mathcal{AR}$ is countable, especially since in practice, we will always be working below some member of $\mathcal{R}$.

The Ellentuck space is a good reference point for digesting this notation. 
In the Ellentuck space, $\mathcal{AR}_n$ is the set of increasing sequences of length $n$ where the entries are natural numbers.
Then  $\iota[\mathcal{R}]$ is the  subset of 
$\prod_{n<\om}\mathcal{AR}_n$ 
consisting of all  infinite sequences of finite sequences which cohere:
For $X=\{x_0,x_1,x_2,\dots\} \in[\om]^{\om}$ enumerated in increasing order, 
\begin{equation}
\iota(X)=\lgl r_n(X):n<\om\rgl=\lgl \emptyset, \{x_0\},\{x_0,x_1\}, \{x_0,x_1,x_2\},\dots\rgl.
\end{equation}
Observe that the sequence on the right recovers $X$ in its limit.

The {\em metric topology}  on $\mathcal{R}$  is the topology generated by basic open sets of the form
$\{X\in\mathcal{R}:r_n(X)=u\}$,
where $n<\om$ and $u\in\mathcal{AR}_n$.
This corresponds to the 
topology on 
$\iota[\mathcal{R}]$  inherited as a subspace of $\mathcal{AR}^{\mathbb{N}}$,
where the countable set $\mathcal{AR}$ has the discrete topology and 
$\mathcal{AR}^{\mathbb{N}}$ has the product (that is, Tychonoff) topology.  
When we speak about $\mathcal{R}$ being {\em closed} in 
$\mathcal{AR}^{\mathbb{N}}$, we mean that the $\iota$-image of $\mathcal{R}$ is a closed subspace of 
$\mathcal{AR}^{\mathbb{N}}$.

The  {\em Ellentuck topology} on $\mathcal{R}$ is the topology generated by the basic open sets
$[u,X]$;
it refines the  metric topology on $\mathcal{R}$.
Given the Ellentuck topology on $\mathcal{R}$,
the notions of nowhere dense, and hence of meager are defined in the natural way.
We say that a subset $\mathcal{X}$ of $\mathcal{R}$ has the {\em property of Baire} iff
 there is an open Ellentuck open set $\mathcal{O}$ such that
 the symmetric difference of $\mathcal{X}$ and $\mathcal{O}$ is meager.

\begin{definition}[\cite{TodorcevicBK10}]\label{defn.5.2}
A subset $\mathcal{X}$ of $\mathcal{R}$ is  {\em Ramsey} if for every $\emptyset\ne[u,X]$,
there is a $Y\in[u,X]$ such that $[u,Y]\sse\mathcal{X}$ or $[u,Y]\cap\mathcal{X}=\emptyset$.
$\mathcal{X}\sse\mathcal{R}$ is {\em Ramsey null} if for every $\emptyset\ne [u,X]$, there is a $Y\in[u,X]$ such that $[u,Y]\cap\mathcal{X}=\emptyset$.

A triple $(\mathcal{R},\le,r)$ is a {\em topological Ramsey space} if every subset of $\mathcal{R}$  with the property of Baire  is Ramsey and if every meager subset of $\mathcal{R}$ is Ramsey null.
\end{definition}

The following result is Theorem
5.4 in \cite{TodorcevicBK10}.

\begin{theorem}[Abstract Ellentuck Theorem]\label{thm.AET}\rm \it
If $(\mathcal{R},\le,r)$ is closed (as a subspace of $\mathcal{AR}^{\bN}$) and satisfies axioms {\bf A.1}, {\bf A.2}, {\bf A.3}, and {\bf A.4},
then every  subset of $\mathcal{R}$ with the property of Baire is Ramsey,
and every meager subset is Ramsey null;
in other words,
the triple $(\mathcal{R},\le,r)$ forms a topological Ramsey space.
\end{theorem}

Rather than repeating the hypotheses of the Abstract Ellentuck Theorem many times throughout this paper, we will simply make the following definition.

\begin{definition}\label{def.axiomatizedtRs}
We say that a topological Ramsey spaces is {\em axiomatized} if it is closed as a subspace of $\mathcal{AR}^{\bN}$ and satisfies axioms {\bf A.1}, {\bf A.2}, {\bf A.3}, and {\bf A.4}.
\end{definition}

\begin{example}\label{Ellspace}
The {\em Ellentuck space} is the triple $([\om]^{\om},\sse,r)$,
where for each $X\in[\om]^{\om}$ and $n<\om$, 
$r_n(X)$ denotes the set of the $n$ least members of $X$.
Here,  $\le_{\mathrm{fin}}$  is simply the subset relation. 
\end{example}

The Ellentuck space is the prototypical topological Ramsey space. 
All known examples of topological Ramsey spaces contain a copy of this space.
Notice that $\sse^*$ is a $\sigma$-closed quasi-order coarsening the partial order $\sse$ on the Ellentuck space.
This forcing $([\om]^{\om},\sse^*)$ produces  a Ramsey ultrafilter on base set $\om$, which is in one-to-one correspondence with $[\om]^1$, also known as the set of first approximations of members of the Ellentuck space.


\subsection{Forcing with   topological Ramsey spaces, and the  ultrafilters they generate}\label{subsec.2.2}

Given a topological Ramsey space $(\mathcal{R},\le,r)$,
there are three  related forcings.
The first is the Mathias-like forcing, where conditions are of the form $[s,A]$ where $s\in\mathcal{AR}$, $A\in\mathcal{R}$ and $s\sqsubset A$.  The second  is $\lgl \mathcal{R},\le\rgl$.
The third is $\lgl \mathcal{R},\le^*\rgl$ where $\le^*$ is some   
 $\sigma$-closed partial order which  coarsens $\le$
 such that 
 the separative quotients of $\lgl \mathcal{R},\le\rgl$ and $\lgl \mathcal{R},\le^*\rgl$ are isomorphic.
These forcings were shown to have many properties in common with Mathias forcing in \cite{Mijares07} and \cite{DiPrisco/Mijares/Nieto17}.
Similarly to $\lgl [\om]^{\om},\sse^*\rgl$, 
forcing with $\lgl \mathcal{R},\le^*\rgl$  adds a new ultrafilter on a countable base set as follows:
By Axiom \bf A.2\rm, relativizing below some member of $\mathcal{R}$ if necessary, the collection  $\mathcal{AR}_1$ of all first approximations to members of $\mathcal{R}$ is a countable set.
We shall let 
\begin{equation}\label{eq.U_G}
\mathcal{U}_{\mathcal{R}}=\{Y\sse\mathcal{AR}_1:\exists X\in G\, (\mathcal{AR}_1\re X\sse Y)\},
\end{equation}
where $G$ is some  generic filter forced by $\lgl\mathcal{R},\le^*\rgl$.
By genericity and the Abstract Ellentuck Theorem \ref{thm.AET}, one sees that $\mathcal{U}_{\mathcal{R}}$ is an ultrafilter on base set $\mathcal{AR}_1$.
For all known topological Ramsey spaces, the collection  $\mathcal{AR}_1$ of first approximations  is a countable set.
If  not, then by Axiom \bf A.2\rm, restricting below some member $Z\in \mathcal{R}$ provides a countable set $\mathcal{AR}_1\re Z$.
If $\lgl\mathcal{R},\le^*\rgl$ is isomorphic to a dense subset of  some $\sigma$-closed  forcing $\mathbb{P}$ which forces a new ultrafilter,  then the ultrafilter forced by $\mathbb{P}$ is isomorphic to $\mathcal{U}_{\mathcal{R}}$. 
 In this way, when we have a forcing $\mathbb{P}$ with a Ramsey space
 as a dense subset, then results for forcing with $\mathbb{P}$ reduce to results
 for ultrafilters forced by Ramsey spaces.

 Ultrafilters $\mathcal{U}_{\mathcal{R}}$ forced by topological Ramsey spaces satisfy  interesting partitition relations.

 \begin{definition}\label{defn.Rdeg}
Given an ultrafilter $\mathcal{U}$ on a countable base set $S$, for each $n\ge 2$,  we define  
$t(\mathcal{U},n)$ to be the least number $t$, if it exists, such that  for any $\ell\ge 2$ and any  coloring $c:[S]^n\ra \ell$, there is a member $U\in\mathcal{U}$ such that $c\re [U]^n$ takes no more than $t$ colors.
When this $t(\mathcal{U},n)$ exists, the standard notation is to write
$$
\mathcal{U}\ra (\mathcal{U})^n_{\ell,t(\mathcal{U},n)},
$$
and  call $t(\mathcal{U},n)$ the {\em Ramsey degree} of $\mathcal{U}$ for $n$.
\end{definition}

Recall  that $\mathcal{U}$ is a Ramsey ultrafilter if and only if 
$t(\mathcal{U},2)=1$ if and only if $t(\mathcal{U},n)=1$ for all $n\ge 2$.
The ultrafilters for which our results hold all have $t(\mathcal{U},2)\ge 2$.

The role  of topological Ramsey spaces in this paper is several-fold.
First, we will need forcings which satisfy analogues of $\om\ra(\om)^{\om}$, and topological Ramsey spaces are natural candidates because of the Abstract Ellentuck Theorem \ref{thm.AET}.
Second, many topological Ramsey spaces  behave  similarly to the Baire space in contexts of AD$_{\mathbb{R}}$, or AD$^+$, or the Solovay model.
Third, in previous papers, many  known forcings producing ultrafilters with interesting Ramsey degrees 
 were shown to contain dense sets forming topological Ramsey spaces (see 
 \cite{DobrinenCreaturetRs16}, 
 \cite{DobrinenJSL16}, 
 \cite{DobrinenJML16},
  \cite{Dobrinen/Mijares/Trujillo17}, \cite{Dobrinen/Todorcevic14},  
 \cite{Dobrinen/Todorcevic15}).
 The original motivation for these constructions was to find  exact  Rudin-Keisler and Tukey structures below these forced ultrafilters, as 
 Ramsey space structure and  Theorem \ref{thm.AET} 
make such results possible.
See \cite{DobrinenSEALS19} for an overview of those results. 
These topological Ramsey spaces  provided a variety of ultrafilters which 
were likely to  have barren extensions.
The topological Ramsey space approach makes the results in this paper possible, while streamlining proofs and making the results applicable to a variety of known forcings.


\section{No new sets of ordinals}\label{sec.barren}

We begin this section by defining  the notion of  extended  coarsened posets.
These are used in the definition of  the Left-Right Axiom,
which 
abstracts the key property of the forcing 
$([\om]^{\om},\sse^*)$  which 
Henle, Mathias, and Woodin 
used in the proof of Theorem \ref{thm.1.0}.
This  axiom along with the assumption  of infinite dimensional partition relations will allow us to prove the more general Theorem \ref{thm.barren}.
 After proving in 
 Lemma \ref{lem.ADimpliesRamsey}
 that, assuming some determinacy, all subsets of a topological Ramsey space are Ramsey,
 we conclude this section with 
 Theorem \ref{thm.barrenufs}.
It follows that a large collection of ultrafilters produce barren extensions, as will be  discussed in Sections \ref{sec.5} and \ref{sec.6}.

In what follows, given  a quasi-order $\le^*$,
 we write $x =^* y$ iff $x \le^* y$ and $y \le^* x$.
 
\begin{definition}\label{poset_def}
A {\em coarsened poset} $\mbb{P}$ is a triple
 $\mbb{P} = \langle X, \le, \le^* \rangle$
  where $\le$ is a partial order and $\le^*$  is a quasi-order  on $X$ with the following  properties:
\begin{enumerate}
\item[1)]
For all $x,y \in X$, $x \le y$ implies 
$x \le^* y$;
\item[2)]
 For all $x,y \in X$,
 if $ y \le^* x$ then there is 
some $ z \le x$ such that $z =^* y$.
 \end{enumerate}
Given $x \in X$, define the following notation:
 $$[x] := \{ y \in X : y \le x \}$$
 $$[x]^* := \{ y \in X : y \le^* x \}.$$
 \end{definition}

\begin{observation}\label{obs.fe}
If  $\mbb{P} = \langle X, \le, \le^* \rangle$ is a coarsened poset, then 
the separative quotients of
 $\lgl X,\le\rgl$ and $\lgl X,\le^*\rgl$ are  isomorphic, so we say that  $\lgl X,\le\rgl$ and $\lgl X,\le^*\rgl$  are {\em forcing equivalent}.
\end{observation}

\begin{proof}
For $x,y\in X$, write $x\|y$ if $x$ and $y$ are compatible in  $\lgl X,\le\rgl$, and write $x\|^*y$ if $x$ and $y$ are compatible in  $\lgl X,\le^*\rgl$.
If $x\|y$, then there is some $z\in X$ such that $z\le x$ and $z\le y$.  Then $z\le^*x$ and $z\le^*y$, so $x\|^*y$.
On the other hand if $x\|^*y$, then there is some $z\in X$ such that $z\le^*x$ and $z\le^*y$.
Using the fact that $z\le^*x$ and letting $z$ play the role of $y$ in 
 2) in Definition \ref{poset_def}, we see that 
there is some $z'\le x$ such that $z'=^* z$.
By 1) and transitivity of $\le^*$, $z'\le^* y$, so by 2) there is some $z''\le y$ such that $z''=^* z'$.
In particular, $z''$ witnesses that $x\|y$.
\end{proof}

A typical  example of a  coarsened poset  has the form 
 $\langle \mc{H}, \subseteq, \subseteq^\mc{I} \rangle$
 where $\mc{I}$ is an ideal on $\omega$ and 
 $\mc{H}$ is  the coideal $\{X\sse\om:X\not\in\mathcal{I}\}$.
Another typical example is  a  topological Ramsey space
 $\lgl \mathcal{R}, \le, \le^*\rgl$, where 
 $\le$ is the partial-order on $\mathcal{R}$ and $\le^*$ is a 
 $\sigma$-closed  quasi-order  coarsening $\le$, where   $\lgl \mathcal{R},\le\rgl$ and $\lgl\mathcal{R},\le^*\rgl$ have the same separative quotient.

\begin{definition}\label{defn.sscp}
Given a  coarsened poset $\mbb{P} = \langle X, \le, \le^* \rangle$ and a quasi-ordered set $\mbb{P}^*=\lgl X^*,\prec\rgl $,
we say that $\mbb{P}^*$  is an {\em equivalent extension} of $\lgl X,\le^*\rgl$ iff the following hold:
\begin{enumerate}
\item[1)]
$X\sse X^*$ and $\prec\re(X\times X)$  equals $\le^*$; and 
\item[2)]
$\lgl X,\le^*\rgl$ is a dense subset of $\lgl X^*,\prec\rgl$.
\end{enumerate}
In this case, we write $\mbb{P}^*=\lgl X,X^*,\le,\le^*\rgl$, and say that $\mbb{P}^*$ is an {\em extended  coarsened poset}, or {\em EC poset}
 (we write $\le^*$ for $\prec$).
Given an EC poset, for $x\in X^*$ define the notation:
$$ 
X^*[x]^*:=\{y\in X^*:y\le^* x\}.
$$
\end{definition}

In some cases $X^*$ will simply be $X$, but for many of our applications in Sections \ref{sec.5} and \ref{sec.6}, we shall need the flexibility of EC posets. 
Notice that 2)  in Definition \ref{defn.sscp} and Observation \ref{obs.fe}
 imply that  
 $\lgl X,\le\rgl$, 
$\lgl X,\le^*\rgl$, and  $\lgl X^*,\le^*\rgl$ are forcing equivalent.

\begin{definition}\label{LRA}
Let $\mbb{P}^* = \langle X, X^*,\le, \le^* \rangle$
 be an extended  coarsened poset.
We say that $\mbb{P}^*$
 \textit{satisfies the Left-Right Axiom (LRA)} 
 iff there are functions
 $\mbox{Left} : X \to X^*$ and
 $\mbox{Right} : X \to X^*$ such that
 the following are satisfied:
\begin{itemize}
\item[1)] 
For each $x \in X$, we have
 $\mbox{Left}(x),
 \mbox{Right}(x) \le^* x$.
 
\item[2)]
 For each $x \in X$,
 there are $y,z \in[x]$
 such that
\begin{itemize}
\item[2a)]
 $\mbox{Left}(y) =^* \mbox{Right}(z)$;
\item[2b)]
 $\mbox{Right}(y) =^* \mbox{Left}(z)$;
\end{itemize}
\item[3)]
Given  $p \in X$,
 for each $x,y \in[p]$,
 there is some $z \in[p]$ such that
\begin{itemize}
\item[3a)] $\mbox{Left}(z) \le^* x$;
\item[3b)] $\mbox{Left}(\mbox{Right}(z))
\le^* x$;
\item[3c)] $\mbox{Right}(\mbox{Right}(z))
\le^* y$;
\end{itemize}
\end{itemize}
\end{definition}

 
 We say that {\em all  cubes  of a poset $\lgl X,\le\rgl$ are Ramsey} if the following holds:
 Given 
 $x \in X$, a positive integer $k$,  and  a coloring
 $c: [x]\to k$,
 there is some $y\le x$ 
 such that $c \restriction [y]$ is constant.

\begin{thm1.3}
Let
$M$ be a transitive  model of ZF.
In $M$, let 
 $\mbb{P} = \langle X, X^*,\le, \le^* \rangle$
 be an extended  coarsened poset  satisfying the
 Left-Right Axiom, and assume that all cubes of $\lgl X,\le\rgl$ are Ramsey.
 Let $N$ be a 
generic extension of $M$ by the forcing
 $\langle X, \le^* \rangle$.
 Then $M$ and $N$ have the same sets of ordinals;
 moreover, every sequence in $N$ of elements of $M$ lies in $M$.
\end{thm1.3}

\begin{proof}
Recall that $\mbb{P}^*$ being an EC poset  means 
that $\lgl X, \le\rgl$, $\lgl X,\le^*\rgl$, and $\lgl X^*,\le^*\rgl$ have isomorphic  separative quotients.
Formally, we shall force with 
 $\langle X^*, \le^* \rangle$, and the forcing relation
 $\forces$ refers to this quasi-order.
 
It suffices to show that given any fixed  $p_0\in X^*$, $\dot{f}$, and ordinal
 $\lambda$ satisfying
 $p_0 \forces \dot{f} :
 \check{\lambda} \to \check{M}$,
there is some $q\in X^*[p_0]^*$
 satisfying
 $q \forces \dot{f} \in \check{M}$.
 We will in fact find such a $q$ in $X$.
Assume towards a contradiction that  for some such
$p_0\in X^*$
 with 
 $p_0 \forces \dot{f} :
 \check{\lambda} \to \check{M}$,
 there is no 
 $q\in X^*[p_0]^*$
such that 
 $q \forces \dot{f} \in \check{M}$.
Then for each
 $p \in X^*[p_0]^*$,
 there  is  a least ordinal
 $\varphi(p) < \lambda$
 such that $p$ does not decide $\dot{f}(\check{\varphi}(p))$; that is, 
 $(\forall u \in M)\,
 p \not\forces \dot{f}(\check{\varphi}(p)) =
 \check{u}$.
Notice that $\varphi$ is invariant,
meaning that whenever $x,y\in X^*$ satisfy  $x=^* y$, then  $\varphi(x)=\varphi(y)$.
 Since $\lgl X,\le^*\rgl$ is dense in $\lgl X^*,\le^*\rgl$,
take some 
 $p_1\in X$ such that $p_1\le^* p_0$.
Define the coloring 
 $c :[p_1] \to 3$  as follows: 
  For $p\in[p_1]$, let 
$$c(p) =
\begin{cases}
 0 & \mbox{if }
   \varphi(\mbox{Left}(p)) <
   \varphi(\mbox{Right}(p)), \\
 1 & \mbox{if }
   \varphi(\mbox{Left}(p)) =
   \varphi(\mbox{Right}(p)), \\
 2 & \mbox{if }
   \varphi(\mbox{Left}(p)) >
   \varphi(\mbox{Right}(p)).
\end{cases}$$
By the hypotheses, there is some 
 $p_2\in[p_1]$ such that  $[p_2]$ is 
homogeneous for $c$;
that is,
 $c \restriction [p_2]$
 is constant.
We claim that $c(p_2) = 1$.

Suppose towards a contradiction
 that $c(p_2) = 0$.
Take  $y, z\in [p_2]$
 satisfying  2a) and 2b) of the Left-Right Axiom.
Since $c(y) =c(p_2)= 0$,
 $$
 \varphi(\mbox{Left}(y)) <
   \varphi(\mbox{Right}(y)).
  $$
Since $\varphi$ is invariant under $=^*$, by
  2a) and 2b) of the LRA, we have
 $$
 \varphi(\mbox{Right}(z))
 =\varphi(\mbox{Left}(y))
 \mathrm{\ \ and\ \ }
  \varphi(\mbox{Right}(y))
   =\varphi(\mbox{Left}(z)).
   $$
Thus,
$$
 \varphi(\mbox{Right}(z))<
   \varphi(\mbox{Left}(z)),
   $$
 so $c(z) = 2$, a 
 contradiction to  
 $c \restriction [p_2]$ being  constant.
 A similar argument shows that $c(p_2)\ne 2$.

Since $p_2$ does not decide the value of
 $\dot{f}(\varphi(\check{p}_2))$,
 there are 
 $x^*,y^*\in X^*[p_2]^*$
 and $u\ne v$ in $M$  such that 
 $x^*
 \forces \dot{f}( \varphi(\check{p}_2) )
 = \check{u}$ and
 $y^*
 \forces \dot{f}( \varphi(\check{p}_2) )
 = \check{v}$.
 Since $\lgl X,\le^*\rgl$ is dense in $\lgl X^*,\le^*\rgl$,
 there are $x',y'\in X$ with $x'\le^* x^*$ and $y'\le^* y^*$; in particular, 
$x',y'\le^* p_2$.
 By 2) of the definition of coarsened poset, there are 
 $x,y\in X$ such that $x\le p_2$ and $x=^* x'$,
 and $y\le p_2$ and $y=^* y'$.
 Thus, we have $x,y\in [p_2]$ such that 
 $x
 \forces \dot{f}( \varphi(\check{p}_2) )
 = \check{u}$ and
 $y
 \forces \dot{f}( \varphi(\check{p}_2) )
 = \check{v}$.
 Fix some  $z \in[p_2]$ satisfying 
 3) of  the LRA
  with regard to $x$ and $y$.
By 3a) we have that
 $\mbox{Left}(z) \le^* x$,
 which implies that 
 $$
 \varphi(\mbox{Left}(z)) \ge
 \varphi(x).
 $$
At the same time,
 $\varphi(x) > \varphi(p_2)$, so
 $$
 \varphi(\mbox{Left}(z)) >
 \varphi(p_2).
 $$
Further,  $\mbox{Right}(z)\le^* z\le p_2$ implies that
 $\varphi(\mbox{Right}(z)) \ge \varphi(p_2)$.
At the same time, 3b) and 3c) of LRA imply that 
 $\mbox{Right}(z)$ is
 $\le^*$-compatible with both $x$ and $y$,
 so 
 $\mbox{Right}(z)$
 does not determine the value of
 $\dot{f}(\varphi(p_2))$;
hence,
 $\varphi(\mbox{Right}(z)) \le \varphi(p_2)$.
 Thus, 
 $$
 \varphi(\mbox{Right}(z))=\varphi(p_2).
 $$
 It follows that 
 $$
 \varphi(\mbox{Left(z)}) > \varphi(p_2)
 = \varphi(\mbox{Right}(z)),
 $$
 implying that  $c(z) = 2$,  contradicting that   $c$ has constant value $1$ on $[p_2]$.
\end{proof}

At this point, let us explain one of the main
 ways to show that every subset of a topological Ramsey space is Ramsey.
 Recall that a subset of a Polish space is Polish if and only if it is $G_{\delta}$. 
We will need the following:

\begin{definition}\label{new_def_projectively_presented}
Let 
$X$ be a Polish space  and $\mathbb{P}=(X,\le)$ be a poset with the property that the subspace $\{(x,y)\in X\times X:x\le y\}$  of $X\times X$ is also Polish.
We say that this space is \defemph{projectively presented} iff
 there is an injection $\eta : X \to \baire$ such that
 the following hold:
\begin{itemize}
\item[1)] 
$\im(\eta)$ is projective.
\item[2)]  
$\{(u,v)\in \im(\eta)\times \im(\eta):\eta^{-1}(u)\le \eta^{-1}(v)\}$
 is projective.
\item[3)] 
Given $p \in  X$ and a function
 $f : [p] \to \baire$
 that is continuous  with respect to the metric topology on $X$, then
 the relation $S \subseteq \baire \times \baire$
 is $\mathbf{\Sigma}^1_1$,
 where 
$$
 S=\{(u,f(\eta^{-1}(u))):u\in \im(\eta)\mathrm{\ and \ }\eta^{-1}(u)\le p\}.
$$
\item[4)]
 For every $p \in X$,
\begin{itemize}
\item[4a)]
for  every
 continuous $f : [p] \to {^\omega 2}$,
 the function $f \circ \eta^{-1}
 : \im(\eta) \to {^\omega 2}$ is $\Sigma^1_1$;
\item[4b)]
 the following set is projective
 (uniformly in $\eta(p)$):
 the set of codes $c \in \baire$
 for $\bf{\Sigma}^1_1$ functions
 $g: \baire \to {^\omega 2}$
 such that
 $g \circ \eta : [p]
 \to {^\omega 2}$ is continuous.
\end{itemize}
\end{itemize}
\end{definition}

Recalling Definition \ref{def.axiomatizedtRs},
we say that a triple
 $\lgl\mathcal{R},\le,r\rgl$ 
is  an  {\em axiomatized topological Ramsey space}  if it
is closed as a subspace of $\mathcal{AR}^{\mathbb{N}}$ and 
 satisfies axioms \bf A.1\rm --\bf A.4. \rm
In most topological Ramsey spaces, the set $\mathcal{AR}$ of all finite approximations of members of the space is countable.
In the unlikely case that it is not, 
 Axiom \bf A.2 \rm guarantees that relativizing   below any member  $p\in \mathcal{R}$, the set $\mathcal{AR}\re p$ is countable.
 Thus, without loss of generality, we shall assume that $\mathcal{AR}$ is countable. 
Assuming  countable choice for sets of reals,
each  axiomatized topological Ramsey space is projectively presented with the following  particularly simple form.

\begin{lemma}\label{lem.tRsprojectivelypres}
Assuming countable choice for sets of reals,
each  axiomatized topological Ramsey
 $\lgl\mathcal{R},\le,r\rgl$  
is projectively presented.
In fact,
there is a   bijection $\eta:\mathcal{R}\ra \baire$  which is  continuous with respect to the metric topology on $\mathcal{R}$ so that 
the set in 2) is analytic, and 
3) and 4) are  true since $\eta$ is  a continuous bijection.
\end{lemma}

\begin{proof}
Either $\mathcal{AR}$ is countable, or else  fix any $p\in\mathcal{R}$ and relativize the proof to $\mathcal{R}\re p$.
For $a\in\mathcal{AR}$, let  
\begin{equation}
E(a)=\{b\in\mathcal{AR}_{|a|+1}:r_{|a|}(b)=a\}.
\end{equation}
By countable choice for reals, 
there is a set of bijections  
\begin{equation}
\eta_a:E(a)\ra \om, \ \  a\in \mathcal{AR}.
\end{equation}
By definition, $\mathcal{AR}_0=\{r_0(x):x\in\mathcal{R}\}$, which has $\emptyset$ as its only member.  
Define $\eta:\mathcal{R}\ra\baire$ as follows:
Given $p\in\mathcal{R}$,
define 
\begin{equation}
\eta(p)=\lgl \eta_{r_n(p)}(r_{n+1}(p)):n<\om\rgl.
\end{equation}
This function $\eta$ is continuous with respect to the metric topology on $\mathcal{R}$, and it is a bijection onto $\baire$; thus, 1) trivially holds.
Furthermore,
the set of all pairs $(p,q)\in \mathcal{R}\times\mathcal{R}$ with $p<q$ is a Polish subspace of $\mathcal{R}\times\mathcal{R}$, since $\mathcal{R}$ is a closed subspace of $\mathcal{AR}^{\mathbb{N}}$.
Thus, the continuous image of this set by $\eta$ is analytic, showing that 2) holds.

To show 3), let $p\in\mathcal{R}$ be fixed, and let $f:[p]\ra\baire$ be a continuous  function.  
Then the relation $S$  in 3) is  certainly $\Sigma^1_1$, since $f\circ \eta^{-1}$ is a continuous function on $[p]$, which is a Polish space since it is a closed subset of $\mathcal{R}$.
Condition 4a) is trivial, since $\eta^{-1}$ is  a continuous bijection.
Likewise, 4b) holds.
\end{proof}

\begin{definition}
{\em $\Sigma^2_1$-reflection} is the statement that
 given any $\Sigma^2_1$ formula,
 if
 $\varphi$ is witnessed by some $A \subseteq \mbb{R}$, then
 $\varphi(A)$ is witnessed by some $A \subseteq \mbb{R}$ that is
 Suslin coSuslin.
\end{definition}

$AD_\mathbb{R}$ implies that every set of reals is Suslin coSuslin.
It is also well known that the axiom $\ad^+ + V = L(\mathcal{P}(\mathbb{R}))$
 implies $\Sigma^2_1$-reflection.
See for example 
\cite{Steel/Trang}
 and   Theorem 25 in \cite{Woodin10}.
We will use this shortly.

The following lemma appears in a modified form in \cite{Feng/Magidor/Woodin92}
 in Theorem 2.2,
 where it is shown that assuming $\zfc$,
 every universally Baire set of reals is Ramsey
 using a countable elementary substructure argument.
Instead of a countable model,
 we use an inner model $M$ such that $\omega_1$
 is inaccessible in $M$.
Given a cardinal $\kappa$ and a tree $T \subseteq {^{<\omega}}(\omega \times \kappa)$,
 let $$p[T] := \{ x \in \baire : (\exists y \in {^{\omega} \kappa}) (\forall n \in \omega)\,
 \langle (x(0),y(0)), ..., (x(n),y(n))\rangle \in T\}.$$
 
\begin{lemma}
Assume there is no injection of $\omega_1$ into $\mathbb{R}$.
Let $A \subseteq [\omega]^\omega$.
Let $e : [\omega]^\omega \to \baire$ be the function that maps each $q \in [\omega]^\omega$ to its
 increasing enumeration.
Assume $A' := \{ e(q): q \in A \} \subseteq \baire$ is Suslin,
 meaning there is a cardinal $\kappa$
 and a tree $T \subseteq {^{<\omega}}(\omega \times \kappa)$
 such that $A' = p[T]$.
Then $A$ is Ramsey.
\end{lemma}

\begin{proof}
Let $[u,q] \subseteq [\omega]^\omega$ be a basic open neighborhood
 in the Ellentuck topology.
We will find a $g \in [u,q]$ such that either
 $[u,g] \subseteq A$ or $[u,g] \cap A = \emptyset$.
Let
 $M = L[T,q]$
 be the inner model generated by $T$ and $q$.
It satisfies the Axiom of Choice
 (because $T$ and $q$ can be coded by a set of ordinals),
 and so since there is no injection of $\omega_1$ into $\mathbb{R}$,
 it must be that $\omega_1$ is inaccessible in $M$.

By the nature of tree representations,
 if $N$ is any inner model which contains $M$, then
 $A' \cap N = (p[T])^N$.
Now let $\mathbb{P} \in M$
 be the Mathias forcing of $M$.
We have $[u,q] \in M$.
Let $\dot{g} \in M$ be the canonical name for the generic real,
 so $1 \forces \dot{g} \in [\omega]^\omega$. 
Consider the statement
 ``$e(\dot{g}) \in p[\check{T}]$''.
Since  $\mathbb{P}$ has the Prikry property,
 there is an $s \in [u,q]$ such that the condition
 $[u,s]$ decides the statement to be either true or false.
Assume for now that
 $$[u,s] \forces e(\dot{g}) \in p[\check{T}].$$

Now since $\omega_1$ is inaccessible in $M$,
 fix a real $g \in [\omega]^\omega$
 that is $\mathbb{P}$-generic over $M$
 such that $g \in [u,s]$.
But one property of Mathias forcing is
 ``the Mathias property''
 (see \cite{Mathias73} and \cite{MathiasThesis}).
In this case, it tells us that every $h \in [u,g]$ (in $V$)
 is $\mathbb{P}$-generic over $M$.
And so the condition $[u,s]$ forces
 each $h \in [u,g]$ (in $V$)
 to be such that $e(h) \in p[T] = A'$,
 so $h \in A$.
Hence,
 $[u,g] \subseteq A$.
 
 The proof for the case that $[u,s]\forces e(\dot{g}) \not\in p[\check{T}]$ is similar.
\end{proof}

If there exists a supercompact cardinal, 
 then  in $L(\mathbb{R})$ 
 every subset of a topological Ramsey space is Ramsey.
This can be shown by considering a topological Ramsey space $\mc{R} \in L(\mathbb{R})$
 and a set $S \subseteq \mc{R}$ in $L(\mathbb{R})$.
Fix a continuous bijection $\eta : \mc{R} \to \baire$ in $L(\mathbb{R})$.
We see that $f(S)$ is $2^\omega$-universally Baire in $V$
 (because large cardinals imply every subset of $\baire$ in $L(\mathbb{R})$
 is $2^\omega$-universally Baire).
Next $f^{-1}(f(S)) = S$ has the property of Baire in $\mc{R}$.
Then we apply the Abstract Ellentuck Theorem \ref{thm.AET}.
In the next theorem, we show that this follows directly from 
either $\ad_\mathbb{R}$ or else $\ad^+ + V=L(\mathcal{P}(\mathbb{R}))$.

\begin{lemma}\label{lem.ADimpliesRamsey}
\label{lem.ADallsetsRamsey}
Assume either 1) $\ad_\mathbb{R}$ or 2) $\ad^+ + V=L(\mathcal{P}(\mathbb{R}))$.
Let $\langle \mc{R}, \le, r \rangle$
 be an  axiomatized topological Ramsey space.
Then every $S \subseteq \mc{R}$ is Ramsey.
\end{lemma}

\begin{proof}
First, suppose there is a counterexample
 $(\langle \mc{R}, \le, r \rangle, \mc{S}).$
Since $\mc{R}$ is axiomatized,
 there is a continuous bijection
 $\eta : \mc{R} \to \baire$ as in
 Lemma~\ref{lem.tRsprojectivelypres}.
Then there must be a counterexample of the form
 $(\langle \mc{R}, \le, r\rangle, S )$
 such that there is a bijection
 $\eta : \mc{R} \to \baire$ such that
\begin{itemize}
\item $\{ (x,y) \in \baire \times \baire: \eta^{-1}(x) = \eta^{-1}(y) \}$ is Suslin,
\item $\{ (x,y) \in \baire \times \baire: \eta^{-1}(x) \le \eta^{-1}(y) \}$ is Suslin,
\item the set coding the $r$ function is Suslin, and
\item $\{ x : \eta^{-1}(x) \in S \}$ is Suslin.
\end{itemize}
Here is why:
 if we have 1), then every set of reals is Suslin.
If we have 2), then by $\Sigma^2_1$ reflection,
 if there were (a set of reals coding) a counterexample, there would be one
 that is Suslin coSuslin.
We will now show that if $\mathcal{Z} \subseteq \baire$
 is a set of reals coding a $(\langle \mathcal{R}, \le, r\rangle, S)$,
 then in fact $S \subseteq \mathcal{R}$ is Ramsey.

We now argue just as in the previous lemma.
Fix $[u,q]$.
We will find a $g \in [u,q]$ such that either
 $[u,g] \subseteq S$ or $[u,g] \cap S = \emptyset$.
Let $\kappa$ be a cardinal and $T \subseteq {}^{<\omega}(\omega \times \kappa)$
 be a tree such that $\mathcal{Z} = p[T]$.
Let $M$ be the inner model $L[T,g]$.
Again $T$ and $g$ can be coded as sets of ordinals.
Both 1) and 2) imply there is no injection of $\omega_1$ into $\mathbb{R}$,
 so $\omega_1$ is inaccessible in $M$.
Note that in any inner model $N$ containing $M$,
 $T$ can be used to talk about $\mc{R} \cap N$ and $\mc{S} \cap N$.
For example, given any $s \in \mc{R} \cap N$,
 $N$ knows whether or not $s \in \mc{S}$.

Let $\mathbb{P}$ be the Mathias forcing
 associated to $\mc{R}$ in $M$.
 Conditions are non-empty basic open sets  $[a,q]$ where $a\in\mathcal{AR}$ and $q\in\mathcal{R}$. 
The ordering is
 $[a,q] \le [b,s]$ iff
 $[a,q] \subseteq [b,s]$.
Let $\dot{g}$ be the name for the generic object, so
 $1 \forces \dot{g} \in \mc{R}$.
The forcing $\mathbb{P}$ has the Prikry property
 (see Theorem 6.7 in \cite{DiPrisco/Mijares/Nieto17}).
So fix $s \in [a,q]$ such that either
 $[a,s] \forces \dot{g} \in \mc{S}$ or
 $[a,s] \forces \dot{g} \not\in \mc{S}$.
Without loss of generality,
 assume the former.

Since $\omega_1$ is inaccessible in $M$,
 fix a $g \in \mc{R}$ that is $\mathbb{P}$-generic over $M$
 such that $g \in [a,s]$.
But $\mathbb{P}$ also has the 
Mathias property
 (see Theorem 6.24 in \cite{DiPrisco/Mijares/Nieto17}).
So every $h \in [a,g]$ (in $V$) is $\mathbb{P}$-generic over $M$.
So $[a,s]$ forces each $h \in [a,g]$ (in $V$)
 to be such that $h \in \mathcal{S}$.
Hence $[a,g] \subseteq \mc{S}$.
\end{proof}

In the following theorem, recall the definition at the end of Subsection \ref{subsec.2.2} of the ultrafilter $\mathcal{U}_{\mathcal{R}}$ forced by $\lgl \mathcal{R},\le^*\rgl$.

\begin{theorem}\label{thm.barrenufs}
Assume that either 1) $\ad_\mathbb{R}$ holds in $V$ or 2) $\ad^+ + V = L(\mathcal{P}(\mathbb{R}))$.
Let $\lgl \mathcal{R}, \le,r\rgl$ be an axiomatized topological Ramsey space and  $\le^*$ be a $\sigma$-closed coarsening of $\mathcal{R}$.
Suppose there is some extended coarsening 
$\lgl \mathcal{R},\mathcal{R}^*,\le,\le^*\rgl$   satisfying the  Left-Right Axiom, and let $\mathcal{U}_{\mathcal{R}}$ be the ultrafilter forced by $\lgl \mathcal{R},\le^*\rgl$  over $V$.
Then $V$ and $V[\mathcal{U}_{\mathcal{R}}]$ have the same sets of ordinals;
 moreover, every sequence in $V[\mathcal{U}_{\mathcal{R}}]$ of elements of $V$ lies in $V$.
\end{theorem}

\begin{proof}
By the previous lemma, all subsets of $\mathcal{R}$ are Ramsey.
Hence  it holds
 that  for
 any $x\in\mathcal{R}$, $k\ge 1$ and coloring $c:[x]\ra k$,  there is some $y\in[x]$ such that $c\re [y]$ is constant. 
 The rest follows from Theorem \ref{thm.barren}.
\end{proof}

The first half  of Theorem \ref{thm.maintRs} follows from Theorem \ref{thm.barrenufs}.


\section{Preservation of Strong Partition Cardinals}\label{sec.4}

In \cite{Henle/Mathias/Woodin85}, 
Henle, Mathias, and Woodin proved that $\mathcal{P}(\om)/\mathrm{fin}$ preserves strong partition cardinals over a model of 
  ZF + EP + LU (Theorem \ref{thm.3.3}).
In this section, we  
extend their result to a wide array of forcings, 
 providing
 conditions which guarantee that  a forcing preserves 
 uncountable strong partition cardinals.

Given a coarsened poset $\langle X,\le, \le^*\rangle$,
a function
 $f$ from $X$ to some other set $Y$ is  called {\em invariant} if and only if 
whenever  $p,q \in X$  satisfy $p =^* q$,
then  $f(p) = f(q)$.
   Similarly, for any  subset $S\sse X$,
 a function 
 $f : S \to Y$  is {\em invariant} if and only if 
whenever  $p,q \in S$  and  $p =^* q$,
then  $f(p) = f(q)$.
We call a set $S \subseteq X$ {\em invariant}
 if and only if
  its characteristic function (from $X$ to $2$) is invariant.
Given a function $f$ whose domain is a subset of $X$\
 (a partial function),
 we call $f$ invariant iff for $p,q \in \dom(f)$ with $p =^* q$,
  then $f(p) = f(q)$.
We call a function $f : X \to Y$ invariant below $p \in X$
 iff $f \restriction [p]$ is an invariant partial function.
A set $S \subseteq X$ is invariant below $p$
 iff $(\forall q_1, q_2 \le p)$
 if $q_1 =^* q_2$, then
 $q_1 \in S \Leftrightarrow q_2 \in S$.

\begin{definition}\label{defn.projpres}
Given $S \subseteq X$ and $p \in X$,
 define
 \begin{align}
 S^+_p &= \{ q \le p :
[q] \subseteq S \}\cr
 S^-_p &= \{ q \le p :[q] \cap S = \emptyset \}.
 \end{align}
We call  $S$ {\em  Ramsey below $p$},
 or simply {\em $R$ below $p$},
 iff $S^+_p \cup S^-_p$ is $\le$-dense below $p$.
We say that  $S$ is  {\em R$^+$ below $p$}
 iff $S^+_p$ is $\le$-dense below $p$, and  $S$  is {\em R$^-$ below $p$}
 iff $S^-_p$ is $\le$-dense below $p$.
 
We shall say that $S$ is {\em Ramsey} iff $S$ is Ramsey below $p$ for each $p\in X$.
Likewise for {\em  $R^+$} and {\em  $R^-$}.
 \end{definition}

 \begin{remark}
 Note that the definition of {\em Ramsey} in Definition \ref{defn.projpres} is weaker than that of Todorcevic in Definition \ref{defn.5.2}.
 As no ambiguity will arise, we use this term  rather than defining yet more terminology. 
 \end{remark}

 Note that if $S$ is invariant and
 $[q] \subseteq S$, then
 $[q]^* \subseteq S$.
 Likewise, if $S$ is invariant and 
 $[q] \cap S = \emptyset$,
 then
 $[q]^* \cap S = \emptyset$.

The following definition of  $\mbox{LU}(\mbb{P})$ extends 
 the Axiom LU in \cite{Henle/Mathias/Woodin85} for 
 $([\om]^{\om},\sse^*)$ to all partial orderings $\mathbb{P}$.

\begin{definition}\label{def.LUP}
Given a poset $\mbb{P} = \langle X, \le \rangle$,
 $\mbox{LU}(\mbb{P})$ is the statement that given
 any relation $R \subseteq X \times
 {^\omega 2}$ and $p \in X$ such that
 $$(\forall x \le p)(\exists y \in {^\omega 2})\,
 R(x,y),$$
 there is some $q \le p$ and some function
 $f : [q] \to {^\omega 2}$
 such that $$(\forall r \le q)\,
 R(r,f(r)).$$
\end{definition}




\begin{observation}
\label{adr_observation}
Let $\mbb{P} = \langle X, \le \rangle$ be a  poset for which  there is an injection
 $\eta : X \to \baire$.  Then the Uniformization Axiom   implies $\mbox{LU}(\mbb{P})$.
\end{observation}

\begin{proof}
Fix $p \in X$
 and a relation $R \subseteq X
 \times {^\omega 2}$
 such that
 $(\forall x \le p)(\exists y \in {^\omega 2})\,
 R(x,y)$.
Consider the relation
 $\tilde{R} \subseteq \baire \times {^\omega 2}$
 defined by $\tilde{R}(\tilde{x},y)$ iff
 either
 $\tilde{x} \not\in \im(\eta)$,
 or
 $R(\eta^{-1}(\tilde{x}),y)$.
By the Uniformization Axiom,
 there is a uniformization
 $\tilde{f} : \baire \to {^\omega 2}$
 of $\tilde{R}$.
This induces a uniformization  function
$f = \tilde{f} \circ \eta$
for $R$.
That is,
 $$(\forall x \le p)\, R(x,f(x)).$$
Thus,  $\mbox{LU}(\mbb{P})$ holds,
 as witnessed by 
 $f \restriction [p]$.
\end{proof}


\begin{definition}
Let $\langle \mc{R}, \le, r \rangle$
 be a topological Ramsey space and let
 $\mbb{P} = \langle \mc{R}, \le \rangle$.
Then $\mbox{LCU}(\mbb{P})$ is  the statement
 $\mbox{LU}(\mbb{P})$, where additionally $f$ is required to
 be continuous with respect to the metric topology
 on $\mc{R}$.
${\mbox{LCU}}^+(\mbb{P})$ is the same statement
 as $\mbox{LCU}(\mbb{P})$ but  replacing 
 ${^\omega 2}$ with $\baire$.
\end{definition}

Certainly $\mbox{LCU}^+(\mbb{P})$ implies
 $\mbox{LCU}(\mbb{P})$.
The other direction holds as well:

\begin{proposition}
\label{lcu_implies_baireclassone}
$\mbox{LCU}(\mbb{P})$ implies $\mbox{LCU}^+(\mbb{P})$.
\end{proposition}
\begin{proof}
Recall  that there is an injection
 $\varphi : \baire \to {^\omega 2}$ such that
 $\varphi^{-1} : \im(\varphi) \to \baire$
 is continuous.
For example, the function $\varphi$ that takes a sequence
 $\langle a_0, a_1, ... \rangle \in \baire$ to the sequence
 $$\overbrace{0 ... 0}^{a_0} 1
   \overbrace{0 ... 0}^{a_1} 1 ....$$
is such a function.
Now consider any $\tilde{R} \subseteq X \times \baire$
 and $p \in X$ such that
 $(\forall x \le p)(\exists y \in \baire)\, \tilde{R}(x,y).$
Define $R \subseteq X \times {^\omega 2}$
 by $(x,y) \in R$ iff
 $y \in \im(\varphi)$ and
 $(x,\varphi^{-1}(y)) \in \tilde{R}$.
Applying $\mbox{LCU}(X)$ to $R$,
 there is some  $q \le p$ and some
 continuous $f : [q] \to {^\omega 2}$
 which  uniformizes $R$ below $q$.
But then $\varphi^{-1} \circ f$
 uniformizes $\tilde{R}$ below $q$ and is
 continuous.
\end{proof}


The following proposition was proved by Mathias \cite{Mathias77}
for  relations of the form 
 $R \subseteq [\omega]^\omega \times \baire$,
 assuming $\omega \rightarrow (\omega)^\omega_2$.
 Todorcevic extended this to  relations of the form
 $R \subseteq [\omega]^\omega \times X$, where
 $R$ is coanalytic and $X$ is an arbitrary Polish space.
 This is stated in \cite{TodorcevicBK10}; a proof appears as
Theorem 7 in \cite{DobrinenCreaturetRs16}, and we follow 
the structure of that  proof.
   First we use
 the hypotheses to find a  uniformization $f$,
 and then perform a fusion argument to find
 a set $[q]$ on which $f$ is continuous.

\begin{proposition}
\label{Mathias_prop}
Let $\langle \mc{R},  \le, r \rangle$
 be a closed axiomatized topological Ramsey space, and let
 $\mbb{P}$ be the poset $\langle \mc{R}, \le \rangle$.
Also assume that either 1) $\ad_\mbb{R}$ holds  or
 2) $\ad^+ + V=L(\mathcal{P}(\mathbb{R}))$ holds.
Then $\mbox{LCU}(\mbb{P})$ holds.
\end{proposition}

\begin{proof}
\noindent \it Claim. \rm
Suppose  $R \subseteq \mc{R} \times {^\omega 2}$ is  a relation and 
 $f : [p^*] \to {^\omega 2}$ is  a uniformization for $R$.
 Then there is  a $q \le p^*$ for which 
 $f \restriction [q]$ is continuous.
\vskip.1in

First, let $p_0 \le p^*$ be such that
 each $q \le p_0$ has the same value
 for $f(q)(0)$.
Such a $p_0$ exists because, by Lemma 
\ref{lem.ADimpliesRamsey},
all subsets of $\mathcal{R}$ are Ramsey, including the set $\{q\le p:f(q)(0)=0\}$.
Let  $s_1 = r_1(p_0)$.
The proof proceeds by induction, recalling the definition of {\em depth} in Section \ref{sec.2} just before Axiom \bf A.3\rm.

Let $n\ge 0$ be fixed and suppose that we have chosen $p_m$ for all $m\le n$ and,  letting $s_{m+1}=r_{m+1}(p_m)$, the following hold for each $0\le m\le n-1$:
\begin{enumerate}
\item
$p_{m+1}\in [s_{m+1},p_m]$; and 
\item
For each $t\le_{\mathrm{fin}} s_{m+1}$ with
$\depth_{p_m}(t)=\depth_{p_m}(s_{m+1})$,
there is a sequence $g_t:m+1 \rightarrow \{0,1\}$ such that 
for each $q\in [t,p_{m+1}]$ and each $k\le m$,
$f(q)(k)=g_t(k)$.
\end{enumerate}
Note that this induction hypothesis is satisfied vacuously for $n=0$, and that for each $m\le n$, $\depth_{p_m}(s_{m+1})=m+1$.

Given $s_{n+1}=r_{n+1}(p_n)$,
let $T$ denote the set of all $t\le_{\mathrm{fin}} s_{n+1}$ for which 
$\depth_{p_n}(t)=\depth_{p_n}(s_{n+1})$.
$T$ is finite, by Axiom \bf A.2 \rm (1).
Let $\triangleleft$  be the ordering of 
 the members of $T$ induced by $\eta$.

Let $t$ be the $\triangleleft$--least member of $T$ for which $p_t$ has not yet been chosen. 
If $t$ is not 
$\triangleleft$--minimum in $T$, then let $u$ denote the $\triangleleft$--predecessor  of  $t$ in $T$.
If  $t$ is 
$\triangleleft$--minimal in $T$, then let $p_u$ denote $p_n$.
For each sequence $g: n+1 \rightarrow 2$, define 
\begin{equation}
\mathcal{X}^{t}_g=\{q\in [t,p_u]:\forall k\le n\, (f(q)(k)=g(k))\}.
\end{equation}
These sets $\mathcal{X}^{t}_g$, $g\in {}^{n+1}2$, form a partition of the basic open set $[t,p_u]$ into finitely many pieces.
Since each piece of the partition is  Ramsey,
there are $q_t\in [t,p_u]$ and  $g_{t}\in {}^{n+1}2$ for which $[t,q_{t}] \subseteq \mathcal{X}^{t}_{g_{t}}$.
By Axiom \bf A.3 \rm (2),
there is some $p_{t}\in [s_{n+1},p_u]$ such that 
$[t,p_{t}] \subseteq [t, q_{t}]$.
At the end of this induction on $(T,\triangleleft)$,
let $p_{n+1}=p_{t^*}$, where $t^*$ denotes the $\triangleleft$--maximum member of $T$.
Note that for each $t\in T$, $p_{n+1} \le p_t$, so in particular,
\begin{equation}
[t, p_{n+1}] \subseteq [t,p_t] \subseteq [t,q_t].
\end{equation}
Thus, for each $q\in [t, p_{n+1}]$ and  $k\le n$, $f(q)(k)=g_t(k)$.
Hence, (1) and (2) hold for $p_{n+1}$.

Let $q=\bigcup_{n\ge 1}s_{n}$.
Since  each $s_{n+1}\sqsupset s_n$ and  $\mathcal{R}$ is a closed topological Ramsey space, it follows that $q$ is a member of $\mathcal{R}$.
We claim that $f$ is continuous on $[q]$.
Suppose $q'\le q$ and $n<\omega$.
Then $f(q')(n)$ is determined by $r_k(q')$, where $k$ is minimal such that $\depth_q(r_k(q'))>n$.
To see this,
let $g: n+1 \rightarrow 2$ be given, and let $N_g$ denote
 $\{h\in {}^{\omega}2:h\restriction (n+1)=g\}$,
 the basic open set in ${}^{\omega}2$ determined by $g$.
Then 
\begin{align}
f^{-1}(N_g)\cap [q] & =
\{q'\le q:f(q)\restriction (n+1)=g\}\cr
&=\bigcup\{[t,q]:t\in\mathcal{AR}| q,\ \depth_q(t)>n,\ \mathrm{and}\ g_t
\restriction (n+1)=g\}, \cr
\end{align}
which is a union of basic open set in the metric topology on $\mathcal{R}$ restricted to $[q]$.
This concludes the proof of the Claim.
\vskip.1in

Supposing  $\ad_\mbb{R}$ holds, 
let $p^*$ in $\mathcal{R}$ be given, and
let $R \subseteq \mc{R} \times {^\omega 2}$
 be a relation with the property that for each $x \le p^*$
 there exists $y \in {^\omega 2}$ such that
 $R(x,y)$.
By the argument in
 Observation~\ref{adr_observation},
 there is a uniformization
 $f : [p^*] \to {^\omega 2}$ for $R$.
Then  the  Claim  implies that 
there  is some $q\le p^*$ such that $f\re[q]$ is continuous. Thus, 
 $\mbox{LCU}(\mbb{P})$  holds.

Now assume $\ad^+ + V=L(\mathcal{P}(\mathbb{R}))$.
 Let  $\eta : \mc{R} \to \baire$ be the continuous bijection defined in 
 Lemma \ref{lem.tRsprojectivelypres}.
Given any relation $\tilde{R} \subseteq \baire \times {^\omega 2}$,
 let $\varphi(\tilde{R})$ be the conjunction
 of the following formulas:
\begin{itemize}
\item $(\forall x \in \baire)(\exists y \in {^\omega 2})\,
 \tilde{R}(x,y)$;
\item $
 (\forall p \in \mathcal{R})
 (\forall \mbox{ continuous } f :
 [p] \to {^\omega 2} )
 (\exists q \le p)\,
 \neg \tilde{R}(\eta(q),f(q))$.
\end{itemize}
By part 4) of
 Definition~\ref{new_def_projectively_presented},
 quantifying over  continuous functions is
 equivalent to quantifying over  reals.
 By Lemma \ref{lem.tRsprojectivelypres},
 $\exists \tilde{R}\,\varphi(\tilde{R})$ is $\Sigma^2_1$.
For any relation  $\tilde{R} \subseteq \baire \times {^\omega 2}$,
 let $N(\tilde{R}) \subseteq \mc{R} \times {^\omega 2}$
 be the relation
 $$N(\tilde{R})(x,y) \Leftrightarrow \tilde{R}(\eta(x),y).$$
Note that for any $\tilde{R}$,
 $\varphi(\tilde{R})$ holds if and only if $N(\tilde{R})$ witnesses the failure
 of $\mbox{LCU}(\mbb{P})$.

Suppose toward a contradiction that  there is some $p^*\in\mathcal{R}$ and some relation  $R' \subseteq
 \mc{R} \times {^\omega 2}$
 which witnesses the failure of $\mbox{LCU}(\mbb{P})$ below $p^*$.
Define $\tilde{R}'$ by
 $$\tilde{R}'(\tilde{x},y) \Leftrightarrow
  R'(\eta^{-1}(\tilde{x}), y).$$
Note that $R' = N(\tilde{R}')$,
 so  $\varphi(\tilde{R}')$ holds.
 By $\Sigma^2_1$-reflection,
 there is a Suslin, co-Suslin set
 $\tilde{R}$ such that $\varphi(\tilde{R})$.
Since $\tilde{R}$ is Suslin, co-Suslin it has
 a uniformization.
Now $R := N(\tilde{R})$ has a uniformization on $[p^*]$ as well.

By the Claim, there is some $q\le p^*$ for which $f\re [q]$ is continuous, contradicting our assumption that $R'$ witnesses the failure of  $\mbox{LCU}(\mbb{P})$.
Therefore,  $\mbox{LCU}(\mbb{P})$ holds. 
\end{proof}

This next definition differs from \cite{Henle/Mathias/Woodin85}
 in that we require the sets to be invariant.

\begin{definition}\label{def.EPP}
Given a coarsened poset $\mbb{P} = \langle X, \le, \le^* \rangle$,
 EP($\mbb{P}$) is the statement that given any $p\in X$ and well-ordered
 sequence $\langle C_\alpha \subseteq X :
 \alpha < \kappa \rangle$
 of sets that are  invariant  and R$^+$
 below  $p$,
 the intersection of the sequence is also
 invariant and R$^+$ below $p$.
\end{definition}

\begin{definition}\label{defn.strcoarse}
We say that  
$\langle \mc{R},  \le, \le^*,  r \rangle$ is a 
{\em coarsened topological Ramsey space} if 
 $\langle \mc{R},  \le,   r \rangle$ is an axiomatized topological Ramsey space and  the following hold:
 \begin{enumerate}
 \item[1)]
 $\le^*$
 is a $\sigma$-closed partial order;
 \item [2)]
 $\langle \mc{R},  \le, \le^*,  r \rangle$ is a coarsened partial order in the sense of Definition \ref{poset_def};
\item[3)]
   Whenever $p,q\in \mathcal{R}$ and there is an $a\in\mathcal{AR}$  satisfying $\emptyset\ne [a,q]\sse [a,p]$, then $q\le^* p$.
\end{enumerate}
\end{definition}

Note that if $\langle \mc{R},  \le, \le^*,  r \rangle$ is a 
 coarsened topological Ramsey space,
then  whenever $S\sse\mc{R}$ is invariant, 
($[p]\sse S\ra [p]^*\sse S)$.

\begin{proposition}\label{prop.ctblintersection}
\label{countable_ep}
Suppose  $\langle \mc{R},  \le, \le^*, r \rangle$
 is a  coarsened   topological Ramsey space.
Let $\langle C_n \subseteq \mc{R} : n < \omega \rangle$
 be a sequence of invariant R$^+$ sets.
Then $\bigcap_n C_n$ is invariant  R$^+$.
\end{proposition}

\begin{proof}
For each $n < \omega$,
 let $D_n := \{ q \in \mc{R} :[q] \subseteq C_n \}$.
Note that each $D_n$ is dense in 
$\lgl \mathcal{R},\le\rgl$ (since $C_n$ is R$^+$)
and is
open in the Ellentuck topology.
Furthermore,  each $D_n$ is R$^+$ and invariant.
Fix $p \in \mc{R}$.
It suffices to find some $q \le p$
 such that $[q] \subseteq \bigcap_n C_n$.

Since $D_0$ is dense, take some $p_0 \le p$  in $D_0$.
Suppose now that  $n<\om$ and  $p_n$ has been chosen.
Since $D_{n+1}$ is open in the Ellentuck topology, by Theorem \ref{thm.AET}
there is some $p_{n+1}\in [r_n(p_n),p_n]$ such that either 
$[r_n(p_n),p_{n+1}]\sse  D_{n+1}$ or else
$[r_n(p_n),p_{n+1}]\cap D_{n+1}=\emptyset$.
The second option cannot happen since $D_{n+1}$ is $R^+$.
Hence, $[r_n(p_n),p_{n+1}]\sse  D_{n+1}$; and in particular,
 $p_{n+1}\in [r_n(p_n),p_n]\cap D_{n+1}$.

 At the end of this process, we have conditions $p_n$ such that $r_0(p_0)\sqsubset r_1(p_1)\sqsubset r_2(p_2)\sqsubset \cdots$.
Since the space $\mathcal{R}$ is closed,
there is a
$q \in \mc{R}$
such that for each $n<\om$, $r_n(q)=r_n(p_n)$.
Then for each $n$, 
$[r_n(q),q]\sse [r_n(q),p_n]$;
so by (3) of Definition \ref{defn.strcoarse},
we have $q \le^* p_n$.
Since each $D_n$ is closed downwards
in $\lgl \mathcal{R},\le\rgl$,
 each $C_n$ is invariant, and  $\langle \mc{R},  \le, \le^* \rangle$ is a coarsened poset, 
 it follows that 
 $[q]^*\sse  \bigcap_{n<\om} C_n$.
 Hence $\bigcap_{n<\om} C_n$ is $R^+$;
 and the intersection of invariant sets is again invariant. 
\end{proof}

\begin{cor}
\label{countable_union_of_rneg_is_rneg}
Let $\langle \mc{R}, \le,\le^*, r \rangle$
 be a
  coarsened closed axiomatized topological Ramsey space.
Let $\langle C_n \subseteq \mc{R} : n < \omega \rangle$
 be a sequence of invariant R$^-$ sets.
Then $\bigcup_n C_n$ is invariant R$^-$.
\end{cor}

\begin{proof}
Apply Proposition \ref{prop.ctblintersection} to the complements
 of the $C_n$'s.
\end{proof}

The proof of the next proposition will use the Kunen-Martin Theorem, which states that given an infinite cardinal $\kappa$ and a wellfounded relation $\prec$ on $\baire$ which is $\kappa$-Souslin as a subset of $\baire\times\baire$,
then $\rho(\prec)<\kappa^+$.
(See Theorem 25.43 on Page 503 in \cite{JechBK}.)

\begin{proposition}\label{prop.EP}
Let $\langle \mc{R}, \le,\le^*, r \rangle$
 be a 
  coarsened  topological Ramsey space.
Assume  $\mbox{LCU}(\langle \mc{R}, \le \rangle)$, countable choice for sets of reals, and 
 every  subset of $ \mc{R}$ is Ramsey.
Then EP$(\mbb{P})$ holds.
\end{proposition}

\begin{proof}
Towards a contradiction,
 assume there is a sequence  of length $\theta$ 
 which witnesses the failure of EP$(\mbb{P})$,
 but EP$(\mbb{P})$ holds for all sequences
 strictly shorter than $\theta$.
By Proposition~\ref{countable_ep},
 it must be that $\theta$ is an uncountable  cardinal, and by minimality of $\theta$ for the failure of EP$(\mbb{P})$, $\theta$ must be regular.
Fix $p \in X$ and
 a sequence $\langle C_\alpha : \alpha < \theta \rangle$
 of invariant subsets of $\mathcal{R}$
 that are R$^+$ below $p$,
 such that $\bigcap_{\alpha<\theta} C_\alpha$ is not 
 R$^+$  below $p$.

For each $\alpha < \theta$,
 let $$D_\alpha := \{
 p \in X :
 (\forall \beta < \alpha)\,
 [p] \subseteq C_\beta \}.$$
Each $D_\alpha$
 is downward $\le$-closed and
 is a subset of $\bigcap_{\beta < \alpha} C_\beta$.
Note that each $D_\alpha$ is invariant,
 because if $[p] \subseteq C_\beta$,
 then $[p]^* \subseteq C_\beta$
 (by the invariance of $C_\beta$ and  since $\le^*$  coarsens $\le$)
 and so any $p' =^* p$ satisfies
 $[p'] \subseteq C_\beta$.
Next, we claim that each $D_\alpha$ is R$^+$ below $p$.
This is immediate from the hypothesis that
 $\bigcap_{\beta < \alpha} C_\beta$ is R$^+$ below $p$.
The sequence
 $\langle D_\alpha : \alpha < \theta \rangle$
 is decreasing.
Let $D_\theta = \bigcap_{\alpha < \theta} D_\alpha$.

Since $\bigcap_{\alpha < \theta} C_\alpha$ is not
 $R^+$ below $p$,
 but is Ramsey (since we are assuming every subset of $\mc{R}$ is Ramsey),
 we may fix a $p' \le p$ such that
 $\bigcap_{\alpha < \theta} C_\alpha$
 and therefore $D_\theta$ is empty below $p'$.
Now define the function
 $\chi : [p'] \to \theta$ as follows:
 $$\chi(q) = \min\{ \alpha < \theta :
 q \not\in D_\alpha \}.$$

Let the continuous bijection $\eta : \mc{R} \to \baire$
 come from  Proposition \ref{lem.tRsprojectivelypres}.
Let $W \subseteq \mc{R} \times \baire$
 be the relation 
$$
 W(x,y')\ \Longleftrightarrow \ \eta^{-1}(y') \le x \mathrm{\  and\ }
 \chi(\eta^{-1}(y')) > \chi(x).
$$
Since $\mbox{LCU}(\langle \mc{R}, \le \rangle)$
 (and therefore $\mbox{LCU}^+(\langle \mc{R}, \le \rangle)$ holds,
 fix $\bar{p} \le p'$ and a
 continuous
 $f : [\bar{p}] \to \baire$
 such that
 $$(\forall r \le \bar{p})\, W(r,f(r)).$$

Now define $S \subseteq \baire \times \baire$ by
$$
 S(x',y')
 \ \Longleftrightarrow \ 
 \eta^{-1}(x') \le \bar{p}\mathrm{\  and\ }
 f(\eta^{-1}(x')) = y'.
$$
Note that 
 $$S(x',y') \Longrightarrow
 \eta^{-1}(x') \ge \eta^{-1}(y') \mbox{ and }
 \chi(\eta^{-1}(x')) < \chi(\eta^{-1}(y')).$$
So, $S$ is  well-founded, meaning there is no sequence
 $\langle x_0', x_1', ... \rangle$ of elements of $\baire$
 such that
 $$... \wedge S(x'_2,x'_1) \wedge S(x'_1,x'_0).$$
Let $D = \{ x' \in \baire :\eta^{-1}(x') \le \bar{p} \}$.
Note that for each $x' \in D$,
 $(\exists y' \in D)\, S(x',y')$.
Since $S$ is a well-founded relation, we may assign an $S$-rank
 $\rho'(x')$ to each $x' \in D$.
Specifically, for $y' \in D$,
 $$\rho'(y') := \sup \{
 \rho(x') + 1 : x' \in D \mbox{ and } S(x',y') \}.$$
For $x' \in \baire$ not in $D$, define $\rho'(x') := -1$.
Since $f$ is continuous,
 by part 3) of
 the Definition~\ref{new_def_projectively_presented}
 of being projectively presented,
 $S$ is a $\mathbf{\Sigma}^1_1$ relation.
Since $S$ is $\mathbf{\Sigma}^1_1$, by the Kunen-Martin theorem,
 fix an ordinal $\gamma < \omega_1$ such that
 $$(\forall x' \in D)\, \rho'(x') < \gamma.$$

For each $0 \le \alpha < \gamma$, let
 $$E'_\alpha := \{ x' \in \baire : \rho'(x) = \alpha \}.$$
Let $E_\alpha := \{ \eta^{-1}(x') : x' \in E'_\alpha \}$.
Note that the $E_\alpha$'s form a partition
 of $[\bar{p}]$,
 because the $E'_\alpha$'s form a partition of $D$.
Here is the second place where we use the assumption that every
 subset of $\mc{R}$ is Ramsey.
We will show that each $E_\alpha$ is R$^-$.
Fix $0 \le \alpha < \gamma$.
The set $E_\alpha$ is Ramsey, so to show it is R$^-$,
 it suffices to show
 $$(\forall q \in E_\alpha)(\exists r \le q)\,
 q' \not\in E_\alpha.$$
This is immediate, because
 given $q \in E_\alpha$,
 there is an $r \le q$ such that
 $S(\eta(q),\eta(r))$.
Thus, by definition of $\rho'$,
 $$\alpha = \rho'(\eta(q)) < \rho'(\eta(r)).$$
Hence, $r \not\in E_\alpha$.

We now have that the $E_\alpha$'s form a partition
 of $[\bar{p}]$, and that they
 are each R$^-$.
By Corollary~\ref{countable_union_of_rneg_is_rneg},
 the countable union of all the $E_\alpha$'s
 is $R^-$.
Hence, $[\bar{p}]$ is R$^-$,
 which is impossible.
\end{proof}

\begin{proposition}\label{merged_prop}
Let $\mbb{P} = \langle X, \le, \le^* \rangle$
 be a coarsened poset such that
 $\mbox{EP}(\mbb{P})$ holds.
Fix $p \in X$.
Let $\langle C_\alpha \subseteq X : \alpha < \kappa \rangle$
 be a sequence of subsets of $X$
 that are Ramsey and invariant below $p$.
Then there is some $q \le p$ such that for each $\alpha < \kappa$,
 either $[q]^* \subseteq C_\alpha$
 or $[q]^* \cap C_\alpha = \emptyset$.
\end{proposition}

\begin{proof}
Fix $\alpha < \kappa$.
Let $D_\alpha \subseteq X$ be the set
 $$D_\alpha = \{ q \le p : [q] \subseteq C_\alpha
 \mbox{ or } [q] \cap C_\alpha = \emptyset \}.$$
Because $C_\alpha$ is Ramsey,
 the set $D_\alpha$ is R$^+$.
Note also that because $C_\alpha$ is invariant, we have
 $$D_\alpha = \{ q \le p : [q]^* \subseteq C_\alpha
 \mbox{ or } [q]^* \cap C_\alpha = \emptyset \}.$$
This also establishes that $D_\alpha$ is invariant.

All we need is a $q$ that is in the intersection
 of the $D_\alpha$'s.
This follows from $\mbox{EP}(\mbb{P})$,
 because each $D_\alpha$ is invariant and R$^+$.
\end{proof}

\begin{proposition}
\label{part_invariantimpliesramsey}
Let $\mbb{P} = \langle X, \le, \le^* \rangle$
 be a  coarsened poset and assume
 $\mbox{EP}(\mbb{P})$.
Assume every $S \subseteq X$ is Ramsey.
Let $\kappa$ be a cardinal,
 $p \in X$,
 and $\Phi : [p] \to [\kappa]^\kappa$
 be an invariant function.
Then  there is a 
 $p' \le p$ such that $ \Phi \restriction [p']$
 is constant.
\end{proposition}

\begin{proof}
For each $\alpha < \kappa$, put
 $C_\alpha := \{ q \le p : \alpha \in \Phi(q) \}$.
Each $C_\alpha$ is Ramsey and invariant below $p$.
By Proposition~\ref{merged_prop},
 fix a $q \le p$
 such that for each $\alpha < \kappa$,
 either $[q]^* \subseteq C_\alpha$ or
 $[q]^* \cap C_\alpha = \emptyset$.
It suffices to show that for each $r \le q$
 that $(\forall \alpha < \kappa)\,
 \alpha \in \Phi(q) \Leftrightarrow \alpha \in \Phi(r)$.
Fix such $\alpha$ and $r$.
We have
 $$\alpha \in \Phi(q) \Rightarrow
 q \in C_\alpha \Rightarrow
 [q]^* \subseteq C_\alpha \Rightarrow
 r \in C_\alpha \Rightarrow \alpha \in \Phi(r)$$
 and
 $$\alpha \not\in \Phi(q) \Rightarrow
 q \not\in C_\alpha \Rightarrow
 [q]^* \cap C_\alpha = \emptyset \Rightarrow
 r \not\in C_\alpha \Rightarrow
 \alpha \not\in \Phi(r).$$
\end{proof}

\begin{proposition}
\label{prop_ctbleqclass}
Suppose $\kappa \rightarrow (\kappa)^\lambda_\mu$,
 where $\kappa, \lambda, \mu$ are non-zero ordinals such that
 $\lambda = \omega \lambda \le \kappa$ and
 $2 \le \mu < \kappa$.
Let $\mbb{P} = \langle X, \le, \le^* \rangle$
 be a coarsened poset with the property that 
each $=^*$ equivalence class is countable, and assume $\mbox{LU}(\mbb{P})$.
Assume there is a surjection $\psi$
 from ${^\omega 2}$ onto $[\kappa]^\kappa$.
Let $\langle \pi_p : p \in X \rangle$
 be a collection of functions
 $\pi_p : [\kappa]^\lambda \to \mu$.
Then below any $p \in X$
 there exists $p^* \le p$
 and an invariant function
 $\Phi : [p^*] \to [\kappa]^\kappa$
 such that $(\forall q \in [p^*])\, \Phi(q)$
 is homogeneous for $\pi_q$.
\end{proposition}

\begin{proof}
Given a  set of ordinals $x$ in  ordertype $\lambda$,
 let ${\Omega(x)}$ be the set of all limits
 of the $\omega$-blocks of $x$.
That is, if $\{ x_\alpha : \alpha < \lambda \}$
 is the increasing enumeration of $x$, then
 $$\Omega(x) = \{
 \sup_{n \in \omega} x_{\omega \beta + n } :
 \beta < \lambda \}.$$
Note that since
 $\omega \lambda = \lambda$,
 ${\Omega(x)}$ is in $[\kappa]^\lambda$
 whenever $x$ is.

For each $q \le p$ define
 $\rho_q : [\kappa]^\lambda \to \mu$ by
 $$
 \rho_q(y) = \pi_q( {\Omega(y)}).
 $$
Let $R \subseteq X \times {^\omega 2}$ be the relation
 $$
 R(q,r) \Longleftrightarrow \psi(r)
 \mbox{ is homogeneous for } \rho_q.
 $$
By  $\mbox{LU}(\mbb{P})$, we may 
 fix a $p^* \le p$ and a function
 $f : [p^*]    \to {^\omega 2}$
 such that
 $$
 (\forall q \le p^*)\, R(q,f(q)).
 $$
 Thus, 
 $$(\forall q \in \mbox{Dom}(f))\, \psi(f(q))
 \mbox{ is homogeneous for } \rho_q.$$
Write $B(q)$ for $\psi(f(q))$.

For each $q \in \mbox{Dom}(f)$, define
 $C(q) \in [\kappa]^\kappa$ as follows:
 $C(q)(0)$ is the least ordinal greater
 than  $B(q')(0)$ for  all $q'$ such that $q' =^* q$.
Let $C(q)(\nu)$ be the least ordinal $\xi$
 such that letting
 $\eta = \sup \{ C(q)(\nu') : \nu' < \nu \}$,
 the interval $[\eta, \xi)$ contains at least
 one element of each $B(q')$ for each $q' =^* q$.
It is in this definition of $C(q)$ that we use
 that each $=^*$ equivalence class is countable.
Without this assumption, we might have
 $C(q)(0) = \kappa$.
Note that if $q' =^* q$ then
 $C(q') = C(q)$.
Hence, $C$ is invariant.

Now for each $q \in \mbox{Dom}(f)$, define
 $\Phi(q) := {\Omega(C(q))}$.
We have that $\Phi$ is invariant.
We claim that $\Phi(q)$ is homogeneous for
 $\pi_q$.
Consider any $x \in [\Phi(q)]^\lambda$.
By construction of $C(q)$, there is some
 $y \in [B(q)]^\lambda$ such that 
 $x = {\Omega(y)}$.
Now $\pi_q(x) = \pi_q({\Omega(y)}) =
 \rho_q(y)$.
$B(q)$ is homogeneous for $\rho_q(y)$,
 so each such value of $\rho_q(y)$ is the same.
Hence, each $x \in [\Phi(q)]^\lambda$
 has the same $\pi_q(x)$ value,
 so $\Phi(q)$ is homogeneous for $\pi_q$.
\end{proof}

\begin{remark}
It is not known if the previous proposition holds 
if the $=^*$ equivalence classes are uncountable.
\end{remark}

\begin{thm1.4}
Suppose $\kappa \rightarrow (\kappa)^\lambda_\mu$,
 where $\kappa, \lambda, \mu$ are non-zero ordinals such that
 $\lambda = \omega \lambda \le \kappa$ and
 $2 \le \mu < \kappa$.
Suppose also that there is a surjection
 from ${^\omega 2}$ to $[\kappa]^\kappa$
 (which happens if we assume $\ad$ and $\kappa < \Theta$).
Let $\mbb{P} = \langle X, \le, \le^* \rangle$
 be a coarsened poset such that
 $\mbox{EP}(\mbb{P})$ and each $=^*$-equivalence class is countable.
Assume every $S \subseteq X$ is Ramsey.
If $\mbox{LU}(\mbb{P})$ holds and $\langle X, \le \rangle$  adds
no new sets of ordinals,
 then $\langle X, \le \rangle$ forces
 $\kappa \rightarrow (\kappa)^\lambda_\mu$.
\end{thm1.4}

\begin{proof}
The relation $\forces$  corresponds to  the forcing
 $\langle X, \le \rangle$.
Note that the assumption
$\kappa \rightarrow (\kappa)^\lambda_{\mu}$
implies  that  $\kappa \rightarrow (\kappa)^\lambda_{\mu+1}$ also holds.
Suppose $p_0 \forces \dot{f} : [\check{\kappa}]^{\check{\lambda}} \to \check{\mu}$.
We will find a $p_2 \le p_0$ in $X$ and some
 $A \in [\kappa]^\kappa$ such that
 $p_2 \forces \dot{f}$ is constant on $[\check{A}]^{\check{\lambda}}$.

For each $p \le p_0$, define a partition
 $\pi_p : [\kappa]^\lambda \to \mu + 1$ by
 $$
 \pi_p(A) =
 \begin{cases} \zeta & \mbox{ if } p \forces \dot{f}(\check{A}) = \check{\zeta}, \\
 \mu & \mbox{ if there is no such } \zeta.
 \end{cases}$$
Using Proposition~\ref{prop_ctbleqclass}
 and assuming $\mbox{LU}(\mbb{P})$,
 there is a $p_1 \le p_0$ and an invariant function
 $\Phi : [p_1] \to [\kappa]^\kappa$
 such that for each $p \le p_1$,
 $\Phi(p)$ is homogeneous for $\pi_p$.
By Proposition~\ref{part_invariantimpliesramsey},
 there is some $p_2 \le p_1$
 with $\Phi$ constant on $[p_2]$.
Set $A = \Phi(p_2)$.
We claim that $$p_2 \forces \check{A} \mbox{ is homogeneous for } \dot{f}.$$
If not, there are $D,E \in [A]^\lambda$, $q \le p_2$
 and $\alpha < \beta < \mu$
 such that
 $$q \forces \dot{f}(\check{D}) = \check{\alpha} \mbox{ and }
 \dot{f}(\check{E}) = \check{\beta}.$$
So, $\pi_q(D) = \alpha$ and $\pi_q(E) = \beta$.
Thus $A$ is not homogeneous for $\pi_q$, contradicting that 
$A = \Phi(q)$, which is homogeneous for $\pi_q$.
\end{proof}

\begin{theorem}\label{thm.4maintRs}
Assume  either AD$_{\mathbb{R}}$ or AD$^+ + V=L(\mathcal{P}(\mathbb{R}))$.
Let $\mbb{P} = \langle \mathcal{R}, \le, \le^* ,r\rangle$
 be a coarsened topological Ramsey space,  where  
 the $=^*$-equivalences classes are countable.
Then forcing with $\lgl \mathcal{R},\le\rgl$
preserves 
 $\kappa \rightarrow (\kappa)^\lambda_\mu$, whenever
  $\kappa \rightarrow (\kappa)^\lambda_\mu$ holds in the ground model,
 where $\kappa, \lambda, \mu$ are non-zero ordinals such that
 $\lambda = \omega \lambda \le \kappa$ and
 $2 \le \mu < \kappa$, and 
 there is a surjection
 from ${^\omega 2}$ to $[\kappa]^\kappa$.
 \end{theorem}

\begin{proof}
Assuming AD,
countable choice for reals holds, so there is a continuous bijection between $\mathcal{R}$ and $\baire$, by Proposition \ref{lem.tRsprojectivelypres}.
Lemma~\ref{lem.ADallsetsRamsey} gives us that every subset of $\mathcal{R}$ is Ramsey.
Then LCU$(\mathbb{P})$ holds by 
Proposition \ref{Mathias_prop}, so Proposition \ref{prop.EP}  implies EP$(\mathbb{P})$ holds.
Theorem \ref{thm.spcpres} yields the result. 
\end{proof}

The second part of Theorem \ref{thm.maintRs} follows 
from Theorem \ref{thm.4maintRs}.

In Section \ref{sec.5} we shall show that three families of topological Ramsey spaces forcing ultrafilters with weak partition properties satisfy the conditions of Theorem \ref{thm.4maintRs}.


 \section{Ultrafilters  with  barren extensions  preserving  strong partition cardinals}\label{sec.5}

 In this and the next section, we  provide  examples of many  forcings  producing  barren extensions  with 
 ultrafilters satisfying  different partition relations.
 Recall Definition
 \ref{defn.Rdeg}:
Given an ultrafilter $\mathcal{U}$ on a countable base set $S$, for each $n\ge 2$, 
$t(\mathcal{U},n)$ is  the least number $t$, if it exists, such that  for any $\ell\ge 2$ and any  coloring $c:[S]^n\ra \ell$, there is a member $U\in\mathcal{U}$ such that $c\re [U]^n$ takes no more than $t$ colors.
An ultrafilter $\mathcal{U}$ is {\em Ramsey} if and only if
 $t(\mathcal{U},n)=1$ for all $n$.

 In this section, we show that three classes of topological Ramsey spaces  forcing  non-Ramsey ultrafilters have  generic extensions with  no new sets of ordinals and preserve strong partition cardinals. 
 These are the class of  Milliken-Taylor ultrafilters  investigated by Mildenberger in \cite{Mildenberger11},
 a hierarchy of ultrafilters of Laflamme  in \cite{Laflamme89} extending weakly Ramsey ultrafilters, and 
 a class of ultrafilters of Dobrinen, Mijares and Trujillo  in \cite{Dobrinen/Mijares/Trujillo17} which encompass  $k$-arrow, non-$(k+1)$-arrow ultrafilters of Baumgartner and Taylor  in \cite{Baumgartner/Taylor78} as well as $n$-square forcing of Blass in \cite{Blass73}.


 \subsection{Milliken-Taylor ultrafilters}\label{subsec.FIN}
 The first class  of coarsened posets that we look at are topological Ramsey spaces of infinite block sequences of vectors.
The reader is referred to Section 5.2 in   \cite{TodorcevicBK10} for a thorough presentation of these  spaces.
The members of $\FIN^{[\infty]}_k$ are infinite sequences $x=\lgl x_i:i<\om\rgl$ such that 
for each $i<\om$, the following hold:
\begin{enumerate}
\item[1)] $x_i$ is a function from $\om$ into $k+1$, and 
 the support of $x_i$,   defined by  $\mbox{supp}(x_i)=\{n\in\om:x_i(n)\ne 0\}$, is finite;
 \item[2)]
 There is some $n\in\mbox{supp}(x_i)$ such that $x_i(n)=k$;
 \item[3)]  
 $\max(\mbox{supp}(x_i))<\min(\mbox{supp}(x_{i+1}))$.
 \end{enumerate}

The $n$-th approximation to $x$ is $r_n(x)=\lgl x_i:i<n\rgl$. 
For $x,y\in\FIN^{[\infty]}_k$,  $y\le x$ iff
$y$ is obtainable from $x$ using the tetris operation.
The  
definition  of the tetris operation 
 is not needed for the proof in this section, so we refer the interested reader to  \cite{TodorcevicBK10}.
For $x\in\FIN^{[\infty]}_k$ and $n<\om$,
let $x/n$ denote the tail $\lgl x_n,x_{n+1},x_{n+2},\dots\rgl$. 
The coarsening  $\le^*$ on $\FIN^{[\infty]}_k$ is defined as follows:
 $y\le^* x$ iff there is some $n$ such that  $y/n\le x$.
 This quasi-order  $\le^*$ is $\sigma$-closed and $\lgl \FIN^{[\infty]}_k,\le\rgl$ and $\lgl \FIN^{[\infty]}_k,\le^*\rgl$ are forcing equivalent.

 $\lgl \FIN^{[\infty]}_k,\le^*\rgl$ forces ultrafilters, referred to as Milliken-Taylor ultrafilters in \cite{Mildenberger11}.
 For $k=1$, such ultrafilters are called stable ordered union ultrafilters and were first investigated by  Blass in \cite{Blass87}.
 In  \cite{Mildenberger11}, Mildenberger showed that  forcing with  $\lgl \FIN^{[\infty]}_k,\le^*\rgl$  produces an ultrafilter, denoted $\mathcal{U}^k$, with at least $k+1$-near coherence classes of ultrafilters Rudin-Keisler below it.

\begin{lemma}\label{lem.FINkLRA}
The coarsened topological Ramsey space  $\lgl \FIN^{[\infty]}_k,\le, \le^*, r\rgl$ satisfies the Left-Right Axiom.
Furthermore, each $=^*$-class is countable.
\end{lemma}

\begin{proof}
In order to show that  the Left-Right Axiom is satisfied by the coarsened topological Ramsey space  $\lgl \FIN^{[\infty]}_k,\le, \le^*, r\rgl$, if suffices to know the following simple fact:
Given any $x\in\FIN^{[\infty]}_k$, both sequences $\lgl x_{2i}:i<\rgl$ and $\lgl x_{2i+1}:i<\om\rgl$ are members of $\FIN^{[\infty]}_k$.
Define the functions Left and Right on $X\in\FIN^{[\infty]}_k$ as follows:
Given $x=\lgl x_i:i<\om\rgl\in X\in\FIN^{[\infty]}_k$,
let 
$\mbox{Left}(x)=\lgl x_{2i}:i<\rgl$ and $\mbox{Right}(x)=\lgl x_{2i+1}:i<\om\rgl$.
Then in the Left-Right Axiom, 1) is satisfied, since  in fact, 
$\mbox{Left}(x)\le x$ and $\mbox{Right}(x)\le x$.
To show that 2) holds,
given  $x\in \FIN^{[\infty]}_k$, let $y=x$ and $z=\lgl x_i:i\ge 1\rgl$.
Then both $y,z\le x$, $\mbox{Left}(y)=^*\mbox{Right}(z)$ since $\mbox{Left}(y)/1=\mbox{Right}(z)$, 
and $\mbox{Right}(y)=\mbox{Left}(z)$.

For 3), given $p\in \FIN^{[\infty]}_k$ and $x,y\le p$,
 define $z$ as follows:
 Take $z_i=x_i$ for $i<3$.
 Then let $z_3 =y_j$ for  $j$ minimal such that $\min(\mbox{supp}(y_j))>\max(\mbox{supp}(x_2))$.
 Given $z_i$ for $i\equiv 0,1,3$ (mod 4),
 let $z_{i+1}=x_j$ for  $j$ minimal such that $\min(\mbox{supp}(x_j))>\max(\mbox{supp}(z_i))$.
 Given $z_i$ for $i\equiv 2$ (mod 4),
 let $z_{i+1}=y_j$ for  $j$ minimal such that $\min(\mbox{supp}(y_j))>\max(\mbox{supp}(z_i))$.
 Then this $z=\lgl z_i:i<\om\rgl$ is in $\FIN^{[\infty]}_k$, $z\le a$, and $z$
 satisfies 3) for $x$ and $y$.
 Thus, the Left-Right Axiom holds.

By the definition of $\le^*$, 
$x=^* y$ if and only if there are $m,n$ such that $x/m=y/n$.
Thus,
each $=^*$-equivalence class is countable.
\end{proof}

 Since $\lgl \FIN^{[\infty]}_k,\le, \le^*, r\rgl$ is a coarsened closed axiomatized topological Ramsey space, it produces a barren extension.

 \begin{cor}\label{cor.Fin}
 Assume  $M$ is a model of ZF +  either 1) AD$_{\mathbb{R}}$ or  2) AD$^+ + V=L(\mathcal{P}(\mathbb{R}))$.
Then  forcing with $\lgl \FIN^{[\infty]}_k,\le^*\rgl$  over $M$ 
 adds an ultrafilter $\mathcal{U}_k$ such that $M[\mathcal{U}_k]$ has the same sets of ordinals as  $M$.
 Furthermore, for  all non-zero ordinals $\kappa,\lambda,\mu$ such that $\lambda=\om\lambda\le\kappa$ and $2\le\mu\le\kappa<\Theta$,
 if $\kappa\ra(\kappa)^{\lambda}_{\mu}$ in $M$, then it also  holds in $M[\mathcal{U}_k]$.
 \end{cor}

 \begin{proof}
 This follows from Theorems  \ref {thm.barrenufs} and \ref{thm.4maintRs} and Lemma \ref{lem.FINkLRA}.
 \end{proof}

 
 \subsection{Extended coarsened posets with Independent Sequencing}\label{subsec.IS}

We define a general property called {\em Independent Sequencing} for partial orders and then for extended coarsened partial orders.  We show that when  an EC poset has Independent Sequencing, then 
the Left-Right Axiom is satisfied and 
hence, Theorem 
 \ref{thm.barren} holds. 
 If further the $=^*$-equivalence classes are countable, then Theorem  \ref{thm.spcpres}
  holds. 
In the following subsections, we show that  the classes of topological Ramsey spaces in the papers
  \cite{DobrinenJSL16},  \cite{DobrinenJML16}, \cite{Dobrinen/Mijares/Trujillo17},  \cite{Dobrinen/Todorcevic14}, and \cite{Dobrinen/Todorcevic15}
  have extended coarsenings which have Independent Sequencing.

\begin{definition}[Independent Sequencing for posets]\label{defn.ISposet}
A poset $\mbb{P}=\lgl X,\le \rgl$ {\em has Independent Sequencing  (IS)} if the following hold:
 \begin{enumerate}
 \item[1)]
For each $x\in X$,   $x$  can be written as a sequence $\lgl x(n):n<\om\rgl$.
Each $x(n)$ is a countable set, possibly with some structure on it.
 \item[2)]
 Given $x,y\in X$, $y\le x$ iff there is a strictly increasing 
 sequence $(i_n)_{n<\om}$ such that  each $y(n)\sse x(i_n)$.
\item[3)]
 Given  $x,y\in X$ and a
 partition of $\om$ into three  pieces,  $P_0,P_1, P_2$, where  at least one of $P_0$ and $P_1$ is infinite,
 there is a $z\in X$ 
 such that for each $n\in P_0$,  $z(n)\sse x(i)$ for some $i$, and for each $n\in P_1$, $z(n)\sse y(i)$ for some $i$.
 Moreover, if   $x,y\le p\in X$, then  there is such a  $z
 \le p$.
 \end{enumerate}
 \end{definition}

This last property 3) is why we call the sequencing ``independent''.

\begin{definition}[Independent Sequencing for EC posets]\label{defn.ISECposet}
An extended coarsened poset $\mbb{P}^*=\lgl X,X^*,\le ,\le^*\rgl$ {\em has Independent Sequencing  (IS)} if  $\lgl X,\le\rgl$ has IS for posets and additionally,
 \begin{enumerate}
 \item[4)]
 $X^*$ is the collection of all  subsequences of members of $X$.
 Thus, 
  $a=\lgl a(n):n<\om\rgl\in X^*$  iff there is an   $x\in X$ and a  strictly increasing sequence $(i_n)_{n<\om}$ such that  each $a(n)=x(i_n)$.
 \item[5)]
 Given $a,b\in X^*$, $b\le a$ iff there is a strictly increasing 
 sequence $(i_n)_{n<\om}$ such that  each $b(n)\sse a(i_n)$.
\item[6)]
 The coarsening $\le^*$ has the property that  for $a,b\in X^*$, 
 if   $\lgl b(n):n\ge m\rgl\le a$ for some $m$, then 
 $b\le^*a$. 
 \end{enumerate}
 \end{definition}
 
 Notice that 4) implies that $X\sse X^*$.  By 2) and 5), the order $\le$ on $X$  is the restriction to $X$ of the order $\le$ on $X^*$. 
 By 3) and 4), for each $a\in X^*$ there is an $x\in X$ such that $x\le a$.
 Thus, $\lgl X,\le\rgl$ is dense in $\lgl X^*,\le\rgl$
 and hence,  $\lgl X,\le^*\rgl$ is dense in $\lgl X^*,\le^*\rgl$.

\begin{lemma}\label{lem.ecIS}
Each  extended  coarsened poset  with   Independent Sequencing satisfies the Left-Right Axiom.
\end{lemma}

\begin{proof}
Let $\mbb{P}^*=\lgl X,X^*,\le ,\le^*\rgl$  be an extended coarsened poset having  Independent Sequencing.
For each $x\in X$, define $\mbox{Left}(x)=\lgl x(2n):n<\om\rgl$ and $\mbox{Right}(x)=\lgl x(2n+1):n<\om\rgl$.
Then  IS 4) implies that Left and Right are functions from $X$ into $X^*$ and  IS  5) implies  that 
 $\mbox{Left}(x), \mbox{Right}(x)\le x$,  so  
LRA  1) holds.
For  LRA 2), given $x\in X$, let $y=x$ and $z=\lgl x(n):n\ge 1\rgl$.
Then 
$$
\mbox{Left}(y)=\lgl x(2n):n<\om\rgl=^*
\lgl x(2n+2):n<\om\rgl=\lgl z(2n+1):n<\om\rgl =\mbox{Right}(z),
$$
 where the $=^*$ holds by  IS 6), so LRA  2a) holds; and 
$$
\mbox{Right}(y)=\lgl y(2n+1):n<\om\rgl=\lgl x(2n+1):n<\om
\lgl z(2n):n\ge 1\rgl=\mbox{Left}(z),
$$
so LRA 2b) holds.

Now let $p,x,y\in X$ with $x,y\le p$ be given. By  IS 3), there is some
 $z\le p$ in $X$  such that for each $n\equiv 0,1,2$ (mod 4), $z(n)=x(i)$ for some $i$, and for each $n\equiv 3$ (mod 4), $z(n)=y(i)$ for some $i$.
 Furthermore, $\mbox{Left}(z)\le x$, $\mbox{Left}(\mbox{Right}(z))\le x$, and $\mbox{Right}(\mbox{Right}(z))\le y$, so  LRA 3) holds.
\end{proof}

Given an extended coarsened topological Ramsey space $\lgl 
\mathcal{R},\mathcal{R}^*,\le,\le^*,r\rgl$,
the forcing $\lgl \mathcal{R},\le^*\rgl$ 
adds an ultrafilter on the countable base set $\mathcal{AR}_1$ as follows:
Letting $G$ be the generic filter on 
$\lgl \mathcal{R},\le^*\rgl$, define  $\mathcal{U}_{\mathcal{R}}$ to be the filter generated by the collection of sets $\mathcal{AR}_1\re X:=\{s\in \mathcal{AR}_1: \exists Y\le X\, (s=r_1(Y))\}$, for $X\in G$.
By genericity and Theorem \ref{thm.AET},
$\mathcal{U}$ is an ultrafilter. 
When the space $\lgl \mathcal{R},\le^*\rgl$ is dense inside some poset $\mathbb{P}$ forcing an ultrafilter, then 
$\mathcal{U}_{\mathcal{R}}$ is isomorphic to the ultrafilter forced by $\mathbb{P}$.
This  fact was behind the work in   
\cite{Dobrinen/Todorcevic14},  \cite{Dobrinen/Todorcevic15}, and 
 \cite{Dobrinen/Mijares/Trujillo17}, which investigated structural results on known ultrafilters  constructed by Laflamme in \cite{Laflamme89}, Baumgartner and Taylor \cite{Baumgartner/Taylor78}, Blass \cite{Blass73}, as well as new ultrafilters related to these.
In the next several subsections, we will show that these classes of ultrafilters have Independent Sequencing and countable $=^*$-equivalence classes, so generic extensions by these ultrafilters will satisfy the next theorem.

The ultrafilters in  Section \ref{sec.6}   are forced by Ramsey spaces constructed in 
 \cite{DobrinenJSL16} and  \cite{DobrinenJML16}, forming a hierarchy over the ultrafilter investigated by 
Szymanski and Xua \cite{Szymanski/Xua83}, forced by $\mathcal{P}(\om\times\om)/\Fin\otimes\Fin$. 
These forcings have Independent Sequencing, but their $=^*$-equivalence classes are uncountable, so the first part of the next theorem holds for them, but we do not know if they preserve strong partition cardinals.

\begin{theorem}\label{thm.ISimpliesnonewsets}
Assume $M$ is a model of ZF + either 1) AD$_{\mathbb{R}}$ or  2) AD$^+ + V=L(\mathcal{P}(\mathbb{R}))$.
In $M$, suppose
 $\lgl \mathcal{R},\mathcal{R}^*,\le,\le^*,r\rgl$ is an extended coarsened topological Ramsey space 
with Independent Sequencing.
Then forcing with $(\mathcal{R},\le^*)$ 
 adds an ultrafilter $\mathcal{U}_{\mathcal{R}}$
such that $M[\mathcal{U}_{\mathcal{R}}]$
has the same sets of ordinals as $M$.
If furthermore, the $=^*$-equivalence classes  are each countable, then $\kappa\ra (\kappa)^{\lambda}_{\mu}$ holds in $M[\mathcal{U}_{\mathcal{R}}]$ whenever $\lambda =\om\lambda\le\kappa$, $ 2\le \mu<\kappa$, there is a surjection from ${}^{\om}2$ to $[\kappa]^{\kappa}$, and $\kappa\ra (\kappa)^{\lambda}_{\mu}$ holds in $M$.
\end{theorem}

\begin{proof}
 This follows from  Lemma \ref{lem.ecIS} and Theorems
 \ref{thm.barren} and \ref{thm.4maintRs}.
\end{proof}

 
 \subsection{A hierarchy of  rapid p-points above a weakly Ramsey ultrafilter}\label{subsec.Laflamme}
 
 Laflamme constructed a sequence of forcings $\mathbb{P}_{\al}$, $1\le \al<\om_1$, in \cite{Laflamme89}  where 
  $\mathbb{P}_{\al}$ forces an ultrafilter $\mathcal{U}_{\al}$ which is a rapid p-point satisfying some weak partition relation.
 Moreover, the Rudin-Keisler structure below $\mathcal{U}_{\al}$ contains a decreasing sequence of order-type $(\al+1)^*$, where the minimal filter is a Ramsey ultrafilter.
 For $k\ge 1$ a finite integer, the Ramsey degrees of $\mathcal{U}_k$  are as follows: For each $n\ge 1$, 
 $t(\mathcal{U}_k,n)=(k+1)^{n-1}$.
 These Ramsey degrees 
 are stated in \cite{Laflamme89} and  succinct proofs using Ramsey theoretic techniques appear in \cite{Dobrinen/NavarroFlores19}.
 In particular,  $t(\mathcal{U}_1,n)=2$, so $\mathcal{U}_1$ is a weakly Ramsey ultrafilter.

In  \cite{Dobrinen/Todorcevic14} and \cite{Dobrinen/Todorcevic15},
for each $1\le\al<\om_1$, a topological Ramsey space 
$\mathcal{R}_{\al}$ was constructed  which is  dense in $\mathbb{P}_{\al}$.
In those papers, these Ramsey spaces were   used to find exact Tukey and Rudin-Keisler structures below each $\mathcal{U}_{\al}$.
Recalling  Theorem \ref{thm.AET},
each $\mathcal{R}_{\al}$ satisfies  the infinite partition relations required in the hypothesis of Theorem \ref{thm.barren}.

The definitions of the spaces $\mathcal{R}_{\al}$ are somewhat  involved and the interested reader is referred to the original papers \cite{Dobrinen/Todorcevic14} and \cite{Dobrinen/Todorcevic15}.  
What is important here is that each member $x\in\mathcal{R}_{\al}$ is a sequence $\lgl x(n):n<\om\rgl$ where each $x(n)$ is a finite tree, and that
 $\lgl \mathcal{R}_{\al},\le\rgl$ has Independent Sequencing.
We define 
$\mathcal{R}_{\al}^*$ to be the set of all sequences $\lgl x(i_n):n<\rgl$ for $x\in \mathcal{R}_{\al}$ and $(i_n)_{n<\om}$ strictly increasing so that IS 4) holds.
Extend $\le$ to $\mathcal{R}_{\al}^*$ so that the first sentence of  IS 5) holds.
The $\sigma$-closed partial order $\le^*$ on $\mathcal{R}_{\al}$ is  simply mod finite initial segment of the sequence, so IS 6) holds.

Given $y\le^* x$ in $X$,
letting $k$ be least such that for each $n\ge k$, $y(n)\sse x(i_n)$ for some $i_n$,
the sequence $z= {r_k(x)}^{\frown} \lgl y(n):n\ge k\rgl$ is a member of $\mathcal{R}_{\al}$.
Furthermore,  this $z\le x$ and $z=^* y$.
Hence $\lgl  \mathcal{R}_{\al},\le,\le^*\rgl$ is a coarsened poset.
We extend this order to $\mathcal{R}_{\al}^*$ to again mean mod finite initial segment. 
Thus, for $a,b\in \mathcal{R}_{\al}^*$, $b\le^* a$ iff for some $k$, for all $n\ge m$, $b(n)\sse a(i)$ for some $i$.
This coarsening satisfies  IS 6).

Thus, 
$\lgl  \mathcal{R}_{\al},\mathcal{R}_{\al}^*,\le,\le^*\rgl$ is an EC poset having IS.
Furthermore, since each $a(n)$ is finite, for $a\in \mathcal{R}_{\al}^*$ and $n<\om$, and since $\le^*$ is mod finite initial segment, it follows that each $=^*$-equivalence class is finite.
Therefore, by Theorem \ref{thm.ISimpliesnonewsets}  Laflamme's forcings produce barren extensions preserving strong partition cardinals, provided that the ground model satisfies  AD$_{\mathbb{R}}$ or  AD$^+ + V=L(\mathcal{P}(\mathbb{R}))$.


\subsection{$k$-arrow ultrafilters, $n$-square ultrafilters, and their extended  family of rapid p-points}\label{subsec.DMT}

Ultrafilters with asymmetric partition relations were constructed by Baumgartner and Taylor in 
\cite{Baumgartner/Taylor78}.
For $k\ge 3$, a  {\em $k$-arrow} ultrafilter is an ultrafilter $\mathcal{U}$ such that for each function $f:[\om]^2\ra 2$, either there is a set $X\in\mathcal{U}$ such that  $f([X]^2)=\{0\}$ or else there is a set $Y\in[\om]^k$ such that $f([Y]^2)=\{1\}$.
This is written as
$$
\mathcal{U}\ra(\mathcal{U},k)^2.
$$
For each $k\ge 3$, Baumgartner and Taylor constructed a partial order, let's call it $\mathbb{P}^{\mathrm{BT}}_k$
which, by
using CH, MA or $\mathfrak{p}=\mathfrak{c}$, 
constructs
 a p-point which is $k$-arrow but not $(k+1)$-arrow.
The partial order  $\mathbb{P}^{\mathrm{BT}}_k$  used finite ordered $k$-clique-free graphs and applications of a theorem of \Nesetril\ and \Rodl, that the collection of finite ordered $k$-clique-free graphs has the Ramsey property \cite{Nesetril/Rodl77}, \cite{Nesetril/Rodl83}.
Recall that a $k$-clique is a complete graph on $k$ vertices, and is denoted by $K_k$.

In \cite{Dobrinen/Mijares/Trujillo17},
 for each $k\ge 3$, a topological Ramsey space $\mathcal{A}_k$ which is dense in Baumgartner and Taylor's partial order $\mathbb{P}^{\mathrm{BT}}_k$ was constructed.
Thus, forcing with $\mathcal{A}_k$  produces a p-point which is $k$-arrow and not $(k+1)$-arrow.
Furthermore, since $\mathcal{A}_k$ is a topological Ramsey space, it satisfies Theorem \ref{thm.AET}, so the desired infinite dimensional partition relation holds. 
To make a space $\mathcal{A}_k$, all that is  required is that we fix  some  sequence $\lgl \mathbb{A}_n:n<\om\rgl$  of finite ordered $K_k$-free graphs such that each $ \mathbb{A}_n$ embeds as an induced subgraph into $\mathbb{A}_{n+1}$ and that each finite ordered $K_k$-free graph embeds  as an induced subgraph into  some $\mathbb{A}_{n}$, and hence into all but finitely many $\mathbb{A}_{n}$.
The members of $\mathcal{A}_k$ are sequences 
$\lgl x(n):n<\om\rgl$ where each $x(n)$ is a subgraph of  
$\mathbb{A}_{i_n}$ for some  strictly increasing sequence $(i_n)_{n<\om}$.
In particular, the topological Ramsey space $\lgl \mathcal{A}_k,\le,r\rgl$ has Independent Sequencing.

The extended coarsening of $\mathcal{A}_k$ is obtained by letting $\mathcal{A}_k^*$ be defined as in 2) of IS. 
The coarsened quasi-order  $\le^*$ is simply mod finite initial segment. 
This order is $\sigma$-closed and $\lgl \mathcal{A}_k, \mathcal{A}_k, ^*,\le,\le^*,r\rgl$ forms an extended coarsened poset with Independent Sequencing.
Furthermore, the $=^*$-equivalence classes are countable, since $\le^*$ is mod finite. 
Thus, Theorem \ref{thm.ISimpliesnonewsets} holds for  Baumgartner and Taylor's p-points which are $k$-arrow and not $(k+1)$-arrow.

Another hierarchy of interesting p-points  are  forced by the  hypercube Ramsey   spaces (see \cite{TrujilloThesis} and \cite{Dobrinen/Mijares/Trujillo17}).
The basis for these spaces is the 
 $n$-square forcing  $\mathbb{P}_{n-\mathrm{square}}$ which Blass constructed in \cite{Blass73}
 in order to show that under MA, there is a p-point which has two Rudin-Keisler incomparable predecessors.
Conditions in 
 $\mathbb{P}_{n-\mathrm{square}}$ are subsets  $p\sse \om\times\om$ such that for each $n\ge 1$, there  are $K,L\in[\om]^n$ such that $K\times L\sse p$; the partial ordering is inclusion mod finite.
A topological Ramsey space forming a dense subset of  $\mathbb{P}_{n-\mathrm{square}}$
was constructed in  \cite{TrujilloThesis}; this was used to show that  both the Rudin-Keisler structure and the Tukey structure below this forced p-point is exactly the four element Boolean algebra, that is, a diamond shape.
 This result was generalized in 
\cite{Dobrinen/Mijares/Trujillo17} where it was shown that for each $k\ge 2$, there is a topological Ramsey space $\mathcal{H}^k$ in which each member is a sequence $x=\lgl x(n):n<\om\rgl$ where each $x(n)$ is a $k$-dimensional cube with side length $n$. 
These were used to show that for each $k\ge 2$, there is a p-point  with both the  initial  Rudin-Keisler structure and the  initial Tukey structure being the  Boolean algbera on $k$ atoms, that is, of cardinality $2^k$.
In particular,
this answered
a  question  about the initial Tukey structure of $\mathcal{G}_2$, which was left open in \cite{Blass/Dobrinen/Raghavan15}.

In fact, $\lgl \mathcal{H}^k,\le\rgl$ as defined in \cite{Dobrinen/Mijares/Trujillo17} has Independent Sequencing.
Defining $X^*$ as  in 4) of Definition \ref{defn.ISECposet} and $\le^*$ to be mod finite initial seqment, then $\lgl  \mathcal{H}^k,X^*,\le, \le^*\rgl$ has Independent Sequencing. 
Furthermore, the $=^*$-equivalence classes are countable, since $\le^*$ is mod finite. 
Thus, Theorem \ref{thm.ISimpliesnonewsets} holds for  Blass'
$n$-square forcing as well as the collection of hypercube topological Ramsey spaces $\mathcal{H}^k$, $k\ge 2$.


 Combining the approaches for the $k$-arrow p-points and the hypercube spaces, the authors of \cite{Dobrinen/Mijares/Trujillo17}  formed a general template for constructing topological Ramsey spaces from countable collections of 
 \Fraisse\ classes. 
 These spaces are formed as follows:
For each $n<\om$,
 let   $J_n\ge 1$  such that either all $J_n$ are equal or else they form an increasing sequence. 
 For each $j<J:=\sup_{n<\om}J_n$, let $\mathcal{K}_j$ be a \Fraisse\ class of finite structures satisfying the Ramsey property. 
For each $j<J$, let $\lgl \mathbb{A}_{n,j}:n<\om\rgl$ be a sequence of members of $\mathcal{K}_j$ such that each member of $\mathcal{K}_j$ embeds into all but finitely many $\mathbb{A}_{n,j}$, and each $\mathbb{A}_{0,j}$ has cardinality one.
Given a sequence $\bar{\mathbb{A}}=\lgl \mathbb{A}_{n,j}:n<\om\rgl$,
a member of the space $\mathcal{R}(\bar{\mathbb{A}})$ is a sequence 
 $x=\lgl x(n):n<\om\rgl$, where for each $n<\om$, $x(n)$ is a sequence $\lgl \mathbb{B}_{n,j}:j<J_n\rgl$ where 
 each $\mathbb{B}_{n,j}$ is a substructure of some $\mathbb{A}_{m,j}$ for $m\ge n$ which is isomorphic to 
 $\mathbb{A}_{n,j}$.
 Each  such topological Ramsey space $\mathcal{R}(\bar{\mathbb{A}})$
  has Independent Sequencing where each $=^*$-class is countable.
  Hence,
  Theorem \ref{thm.ISimpliesnonewsets} holds for  Blass'
$n$-square forcing as well as the collection of hypercube topological Ramsey spaces $\mathcal{H}^k$, $k\ge 2$.
  Moreover, the 
   forced p-points have the following  interesting property:
 Given $\mathcal{U}_{\bar{\mathbb{A}}}$ the forced p-point from $\mathcal{R}(\bar{\mathbb{A}})$,
 its initial Rudin-Keisler structure is isomorphic to the embedding structure of the sequence of \Fraisse\ classes, whereas its initial Tukey structure is either a finite Boolean algebra (if $J$ is finite) or else has the structure of $([\om]^{<\om},\sse)$ if $J=\om$.


 \section{A hierarchy of  non-p-points  with barren extensions}\label{sec.6}

 The forcing $([\om]^{\om},\sse^*)$ produces a Ramsey ulftrafilter, which  we shall denote  by $\mathcal{G}_1$ in this section.
 Recall that the separative quotient of $([\om]^{\om},\sse^*)$  is the collection of  non-zero elements of the Boolean algebra $\mathcal{P}(\om)/\Fin$.
There is 
natural hierarchy of Boolean algebras, $\mathcal{P}(\om^{k})/\Fin^{\otimes k}$, $k\ge 2$, extending 
 $\mathcal{P}(\om)/\Fin$.  
 In fact, an even more general collection of Boolean algebras 
 $\mathcal{P}(B)/\Fin^{\otimes B}$, 
where $B$ is a uniform barrier of countable rank,
 was constructed in \cite{DobrinenJML16}; 
 these   can be thought of as very precise means of forming 
  $\mathcal{P}(\om^{\al})/\Fin^{\otimes \al}$,  for all  $1\le \al<\om_1$. 
 This collection of Boolean algebras  differs substantially  from the hierarchies of forcings   described in the previous section: in particular, these forced  ultrafilters are not p-points and their $=^*$-equivalence classes are not countable. 
 Nevertheless, we will see that there are    extended coarsened partial orders 
 having Independent Sequencing which are forcing equivalent to these Boolean algebras, so they will all produce barren extensions.

The ideal $\Fin^{\otimes 2}$  consists of those sets $A\sse\om\times\om$ such that for all but finitely many $n<\om$, the $n$-th fiber $\{j<\om:(n,j)\in A\}$ is finite. 
$\Fin^{\otimes 2}$  is a $\sigma$-ideal under the quasi-order $\sse^{\Fin^{\otimes 2}}$
and the Boolean algebra $\mathcal{P}(\om\times\om)/\Fin^{\otimes 2}$ forces an ultrafilter $\mathcal{G}_2$ which is not a p-point but in the terminology of \cite{Blass/Dobrinen/Raghavan15} is the {\em next best thing to a p-point} in the following sense: $t(\mathcal{G}_2, 2)=4$.
This is the strongest partition property that a non-p-point can have, since  any ultrafilter satisfying 
$t(\mathcal{U},2)=3$  actually is a p-point. 
Furthermore,   the projection of  $\mathcal{P}(\om\times\om)/\Fin^{\otimes 2}$ to its first coordinate recovers 
$\mathcal{P}(\om)/\Fin$, and  this projection 
of members of $\mathcal{G}_2$  recovers a Ramsey ultrafilter  on $\om$.
Properties of this ultrafilter $\mathcal{G}_2$  have been studied in 
\cite{Szymanski/Xua83}, 
	\cite{Hrusak/Verner11},
	\cite{Blass/Dobrinen/Raghavan15},
	and \cite{DobrinenJSL16}.
	The only non-principal ultrafilter Rudin-Keisler strictly below $\mathcal{G}_2$ is exactly the Ramsey ultrafilter $\pi_1(\mathcal{G}_2)$, or any ultrafilter isomorphic to it
(Corollary 3.9 in \cite{Blass/Dobrinen/Raghavan15}).
Thus, $\mathcal{G}_2$ is  RK-minimally more complex than a Ramsey ultrafilter.

The  construction of $\Fin^{\otimes 2}$ from $\Fin$ can be extended recursively to obtain $\sigma$-closed ideals   on $\om^k$ as well.
Given $k\ge 2$ and the $\sigma$-closed ideal $\Fin^{\otimes k}$ on $\om^k$,
define $A\sse \om^{k+1}$ to be a member of $\Fin^{\otimes k+1}$  iff for all but finitely many $n<\om$, the $n$-th fiber $\{\vec{j}\in \om^k:(n)^{\frown}\vec{j}\in A\}$ is in $\Fin^{\otimes k}$.
This produces  a hierarchy of Boolean algebras  with the property  that for any $1\le j<k<\om$,
projecting the members of  $\mathcal{P}(\om^k)/\Fin^{\otimes k}$ to the first $j$ coordinates recovers  $\mathcal{P}(\om^j)/\Fin^{\otimes j}$.
Likewise, given an ultrafilter $\mathcal{G}_k$ forced by
$\mathcal{P}(\om^k)/\Fin^{\otimes k}$,
projecting its members to the first $j$ coordinates produces an ultrafilter on $\om^j$ which is generic for 
$\mathcal{P}(\om^j)/\Fin^{\otimes j}$.
A  formula for the Ramsey degrees of these ultrafilters was found by Navarro Flores and appears in 
\cite{Dobrinen/NavarroFlores19}: For each $1\le k<\om$, 
$$
t(\mathcal{G}_k,2)= \sum_{i=0}^{k-1} 3^i.$$

Moving to the countable transfinite,
similarly to countable iterations of  Fubini products of ultrafilters, 
  there are choices to be made in deciding how to define $\Fin^{\otimes \al}$ for $\om\le \al< \om_1$.
However, if one works with barriers, the construction is concrete.
This paper will not go into the definition  and theory of barriers, but refers the interested reader to \cite{Argyros/TodorcevicBK} for an introduction to this area.
Suffice it  to mention here that given a uniform barrier $B$ on $\om$, the  
order type of 
 $B$ with its  lexicographic order  is  some countable ordinal, say $\al_B$, and every countable ordinal is achievable in this way. 
The recursive construction of the ideals continues by recursion on the rank of the barrier  to form $\Fin^{\otimes B}$.
Each $B$  produces a different Boolean algebra $\mathcal{P}(B)/\Fin^{\otimes B}$, which force interesting ultrafilters $\mathcal{G}_B$ with the property that if $B$ projects to a barrier $C$, then $\mathcal{G}_B$  has a copy of $\mathcal{G}_C$ Rudin-Keisler below it. 
If $\al_B$ is infinite, then for any $1\le k<\om$, the projection of  $\mathcal{G}_B$  to its first $k$ coordinates reproduces $\mathcal{G}_k$. 
 We point out that the   forcing properties of 
 a related hierarchy was studied by Kurili\'{c} in \cite{Kurilic15}; his hierarchy  agrees with the one here    for $k$ finite, but differs for  $\al\ge \om$.

In 
 \cite{DobrinenJSL16} 
 and \cite{DobrinenJML16},
 topological Ramsey spaces  were constructed forming 
 dense subsets of  $\mathcal{P}(\om^k)/\Fin^{\otimes k}$ for $2\le k<\om$ and $\mathcal{P}(B)/\Fin^{\otimes B}$ for all uniform barriers $B$ of countable rank, respectively.
 These Ramsey spaces are denoted $\mathcal{E}_k$ and $\mathcal{E}_B$, respectively;
 here we shall simply use $\mathcal{E}_B$ as it subsumes the former case. 
 The Ramsey structure was utilized in  \cite{DobrinenJSL16}  to prove that both the initial Rudin-Keisler and initial Tukey structures below the forced ultrafilter $\mathcal{G}_k$  is exactly  a chain of length $k$. 
 (The extension of this to the broader collection of $\mathcal{G}_B$  for $B$ a barrier is in preparation.)
 In particular, it was shown that $\mathcal{G}_2$ is Tukey-minimal above the projected Ramsey ultrafilter $\pi_1(\mathcal{G}_2)$,  answering a question that was left open in \cite{Blass/Dobrinen/Raghavan15}.

Refraining from going into detail about these spaces here,
what is important for this paper is that given a uniform barrier $B$, 
 each member of  the space $\mathcal{E}_B$ may be regarded as an infinite sequence $x=\lgl x(n):n<\om\rgl$  
 such that each $x(n)$ is a fiber of $B$ (so infinite) and that $\lgl\mathcal{E}_B,\le\rgl$ has Indpendent Sequencing.
 Furthermore, 
 letting $\mathcal{E}_B^*$ be defined as in 4) of IS and $\le^*$ be defined by 
 $b\le^*a$ iff for all but finitely many $n$, $b(n)\sse a(i)$ for some $i$,
 then 
 $\lgl \mathcal{E}_B,\mathcal{E}_B^*,\le,\le^*\rgl$
 is an extended coarsened poset having Independent Sequencing.
 The forcings $\lgl \mathcal{E}_B,\le\rgl$, $\lgl \mathcal{E}_B^*,\le^*\rgl$ and  $\lgl (\Fin^{\otimes B})^+,\sse^{\Fin^{\otimes B}}\rgl$ each have separative quotient which is isomorphic to the non-zero members of $\mathcal{P}(B)/\Fin^{\otimes B}$, hence all force the same ultrafilter, $\mathcal{G}_B$.
The properties of the spaces that make this true are contained in the work in \cite{DobrinenJSL16} and \cite{DobrinenJML16}.
Thus, Theorem \ref{thm.ISimpliesnonewsets} yields that forcing with $\mathcal{P}(B)/\Fin^{\otimes B}$ adds no new sets of ordinals in $M[\mathcal{G}_B]$, assuming  $M$ satisfies ZF and either AD$_{\mathbb{R}}$ or
 AD$^+$ + $V = L(\mathcal{P}(R))$.


\section{Further directions and open problems}\label{section.end}

As forcings, topological Ramsey spaces  have so many characteristics in common with $\mathcal{P}(\om)/\Fin$ that a natural  goal is  
to find out  exactly
how far these similarities persist.
One direction is preservation of strong partition cardinals.
In Section \ref{sec.6}, we showed that the forcings $\mathcal{P}(\om^{\al})/\Fin^{\otimes \al}$, $2\le \al<\om_1$, produce barren extensions. 
However, our methods do not prove that these forcings preserve strong partition cardinals.

\begin{question}
Does forcing with $\mathcal{P}(\om\times\om)/\Fin^{\otimes 2}$ preserve strong partition cardinals?  If so, does the same hold for  each $\mathcal{P}(\om^{\al})/\Fin^{\otimes \al}$, $2\le \al<\om_1$?
\end{question}

If so, then the following question is worth pursuing, as  Navarro Flores  has shown that each topological Ramsey space adds a new ultrafilter with a Ramsey ultrafilter Rudin-Keisler below it \cite{NavarroFloresThesis}. 

\begin{question}
Does forcing with any topological Ramsey space preserve strong partition cardinals?  
\end{question}

 Recently,
 Zheng proved  in \cite{Zheng17} and \cite{Zheng18} that the
  ultrafilters  considered in Subsections \ref{subsec.FIN} and \ref{subsec.Laflamme} are preserved under side-by-side Sacks forcing with countable support,
  and it follows from work in  \cite{Zheng18} that  this also holds for the ultrafilters considered in \ref{subsec.DMT}.
  In \cite{Zheng20}, Zheng also proved that ultrafilters forced by $\mathcal{P}(\om^k)/\Fin^{\otimes k}$ are preserved  under side-by-side Sacks forcing with countable support, and her methods should also hold for the whole hierarchy into countable ordinals, using work in \cite{DobrinenJML16}.

 \begin{question}
 Is there a connection between  an ultrafilter being preserved by side-by-side Sacks forcing and having a  barren extensions? 
 Does one imply the other?
 \end{question}

 By 
Corollary 5.3 in \cite{DiPrisco/Todorcevic98},
if $L(\mathbb{R})$ is a Solovay model,
then
the $\mathcal{P}(\omega)/\Fin$ extension of $L(\mathbb{R})$
 satisfies the perfect set property.
So we ask the following:

\begin{question}
Assume AD$_\mathbb{R}$ or AD$^+ + V = L(\mathcal{P}(\mathbb{R}))$.
Which   topological Ramsey spaces 
 preserve  the perfect set property via forcing?
\end{question}

In particular, we conjecture that topological Ramsey spaces with Independent Sequencing will preserve the perfect set property.
It may  well be the case that all topological Ramsey spaces behave like $([\om]^{\om},\sse^*)$ in this and many more respects.

\bibliographystyle{amsplain}
\bibliography{referencesDH_revision}

\providecommand{\bysame}{\leavevmode\hbox to3em{\hrulefill}\thinspace}
\providecommand{\MR}{\relax\ifhmode\unskip\space\fi MR }
\providecommand{\MRhref}[2]{%
  \href{http://www.ams.org/mathscinet-getitem?mr=#1}{#2}
}
\providecommand{\href}[2]{#2}
\begin{thebibliography}{10}

\bibitem{Argyros/TodorcevicBK}
Spiros~A. Argyros and Stevo Todorcevic, \emph{Ramsey {M}ethods in {A}nalysis},
  Birkh{\"{a}}user, 2005.

\bibitem{Baumgartner/Taylor78}
James~E. Baumgartner and Alan~D. Taylor, \emph{Partition {T}heorems and
  {U}ltrafilters}, Tansactions of the American Mathematical Society
  \textbf{241} (1978), 283--309.

\bibitem{Blass73}
Andreas Blass, \emph{The {R}udin-{K}eisler ordering of {P}-{P}oints},
  Transactions of the American Mathematical Society \textbf{179} (1973),
  145--166.

\bibitem{Blass87}
\bysame, \emph{Ultrafilters related to {H}indman's finite-unions theorem and
  its extensions}, Contemporary Mathematics \textbf{65} (1987), 89--124.

\bibitem{Blass/Dobrinen/Raghavan15}
Andreas Blass, Natasha Dobrinen, and Dilip Raghavan, \emph{The next best thing
  to a p-point}, Journal of Symbolic Logic \textbf{80}, no.~3, 866--900.

\bibitem{Carlson/Simpson90}
Timothy~J. Carlson and Stephen~G. Simpson, \emph{Topological {R}amsey theory},
  Mathematics of {R}amsey theory{\rm, volume 5 of} {A}lgorithms and
  {C}ombinatorics, Springer, 1990, pp.~172--183.

\bibitem{DiPrisco/Mijares/Nieto17}
Carlos DiPrisco, Jos{\'{e}}~Grigorio Mijares, and Jesus Nieto, \emph{Local
  {R}amsey theory: an abstract approach}, Mathematical Logic Quarterly
  \textbf{63} (2017), no.~5, 384--396.

\bibitem{DiPrisco/Todorcevic98}
Carlos~Augusto DiPrisco and Stevo Todorcevic, \emph{Perfect-{S}et {P}roperties
  in $l(\mathbb{R}[u]$}, Advances in Mathematics \textbf{139} (1998), no.~2,
  240--259.

\bibitem{DobrinenCreaturetRs16}
Natasha Dobrinen, \emph{Creature forcing and topological {R}amsey spaces},
  Topology and Its Applications \textbf{213} (2016), 110--126, Special issue in
  honor of Alan Dow's 60th birthday.

\bibitem{DobrinenJSL16}
\bysame, \emph{High dimensional {E}llentuck spaces and initial chains in the
  {T}ukey structure of non-p-points}, Journal of Symbolic Logic \textbf{81}
  (2016), no.~1, 237--263.

\bibitem{DobrinenJML16}
\bysame, \emph{Infinite dimensional {E}llentuck spaces and
  {R}amsey-classification theorems}, Journal of Mathematical Logic \textbf{16}
  (2016), no.~1, 37 pp.

\bibitem{DobrinenSEALS19}
\bysame, \emph{Topological {R}amsey spaces dense in forcings}, Structure and
  Randomness in Computability and Set Theory, World Scientific, 2020, p.~32 pp.

\bibitem{Dobrinen/Mijares/Trujillo17}
Natasha Dobrinen, Jos{\'{e}}~G. Mijares, and Timothy Trujillo,
  \emph{Topological {R}amsey spaces from {F}ra{\"{i}}ss{\'{e}} classes,
  {R}amsey-classification theorems, and initial structures in the {T}ukey types
  of p-points}, Archive for Mathematical Logic, special issue in honor of James
  E. Baumgartner \textbf{55} (2017), no.~7-8, 733--782, (Invited submission).

\bibitem{Dobrinen/NavarroFlores19}
Natasha Dobrinen and Sonia Navarro~Flores, \emph{Ramsey degrees of
  ultrafilters, pseudointersection numbers, and the tools of topological
  {R}amsey spaces},  (2019), 25 pp, Submitted.

\bibitem{Dobrinen/Todorcevic14}
Natasha Dobrinen and Stevo Todorcevic, \emph{A new class of
  {R}amsey-classification {T}heorems and their applications in the {T}ukey
  theory of ultrafilters, {P}art 1}, Transactions of the American Mathematical
  Society \textbf{366} (2014), no.~3, 1659--1684.

\bibitem{Dobrinen/Todorcevic15}
\bysame, \emph{A new class of {R}amsey-classification {T}heorems and their
  applications in the {T}ukey theory of ultrafilters, {P}art 2}, Transactions
  of the American Mathematical Society \textbf{367} (2015), no.~7, 4627--4659.

\bibitem{Ellentuck74}
Erik Ellentuck, \emph{A new proof that analytic sets are {R}amsey}, Journal of
  Symbolic Logic \textbf{39} (1974), no.~1, 163--165.

\bibitem{Feng/Magidor/Woodin92}
Qi~Feng, Menachem Magidor, and Hugh Woodin, \emph{Universally baire sets of
  reals}, Set {T}heory of the {C}ontinuum. Mathematical Sciences Research
  Institute Publications (Woodin~H. Judah~H, Just~W., ed.), vol.~26,
  North-Holland, 1992, pp.~203--242.

\bibitem{Galvin/Prikry73}
Fred Galvin and Karel Prikry, \emph{Borel sets and {R}amsey's {T}heorem},
  Journal of Symbolic Logic \textbf{38} (1973), 193--198.

\bibitem{Henle/Mathias/Woodin85}
James~M. Henle, Adrian R.~D. Mathias, and W.~Hugh Woodin, \emph{A barren
  extension}, Methods in mathematical logic, Lecture Notes in Mathematics,
  1130, Springer, 1985.

\bibitem{Hrusak/Verner11}
Michael Hru{\v{s}}{\'{a}}k and Jonathan~L. Verner, \emph{Adding ultrafilters by
  definable quotients}, Rend. Circ. Mat. Palermo \textbf{60} (2011), 445--454.

\bibitem{JechBK}
Thomas Jech, \emph{Set theory. {T}he {T}hird {M}illennium {E}dition, {R}evised
  and {E}xpanded.}, Springer, 2003.

\bibitem{Kurilic15}
Milo{\v{s}}~S. Kurili{\'{c}}, \emph{Forcing with copies of countable ordinals},
  Proceedings of the American Mathematical Society \textbf{143} (2015), no.~4,
  1771--1784.

\bibitem{Laflamme89}
Claude Laflamme, \emph{Forcing with filters and complete combinatorics}, Annals
  of Pure and Applied Logic \textbf{42} (1989), 125--163.

\bibitem{MathiasThesis}
A.~R.~D. Mathias, \emph{On a generalisation of {R}amsey's theorem}, Ph.D.
  thesis, University of Cambridge, 1968.

\bibitem{Mathias73}
\bysame, \emph{On sequences generic in the sense of {P}rikry}, Journal of the
  Australian Mathematical Society \textbf{15} (1973), 409--414.

\bibitem{Mathias77}
\bysame, \emph{Happy families}, Annals of Mathematical Logic \textbf{12}
  (1977), no.~1, 59--111.

\bibitem{Mijares07}
Jos{\'{e}}~G. Mijares, \emph{A notion of selective ultrafilter corresponding to
  topological {R}amsey spaces}, Mathematical Logic Quarterly \textbf{53}
  (2007), no.~3, 255--267.

\bibitem{Mildenberger11}
Heike Mildenberger, \emph{On {M}illiken-{T}aylor ultrafilters}, Notre Dame
  Journal of Formal Logic \textbf{52} (2011), no.~4, 381--394.

\bibitem{NashWilliams65}
C.~St. J.~A. Nash-Williams, \emph{On well-quasi-ordering transfinite
  sequences}, Proceedings of the Cambridge Philosophical Society \textbf{61}
  (1965), 33--39.

\bibitem{NavarroFloresThesis}
Sonia Navarro~Flores, \emph{Topological {R}amsey space ideals}, Ph.D. thesis,
  Universidad Nacional Autónoma de M{\'{e}}xico, 2020 expected.

\bibitem{Nesetril/Rodl77}
Jaroslav Ne{\v{s}}et{\v{r}}il and Vojt{\v{e}}ch R{\"{o}}dl, \emph{Partitions of
  finite relational and set systems}, Journal of Combinatorial Theory Series A
  \textbf{22} (1977), no.~3, 289--312.

\bibitem{Nesetril/Rodl83}
\bysame, \emph{Ramsey classes of set systems}, Journal of Combinatorial Theory
  Series A \textbf{34} (1983), no.~2, 183--201.

\bibitem{Prikry76}
Karel Prikry, \emph{Determinateness and partitions}, Proceedings of the
  American Mathematical Society \textbf{54} (1976), no.~1, 303--306.

\bibitem{Shelah/Woodin90}
Saharon Shelah and Hugh Woodin, \emph{Large cardinals imply that every
  reasonably definable set of reals is {L}ebesgue measurable}, Israel Journal
  of Mathematics \textbf{70} (1990), 381--384.

\bibitem{Silver71}
Jack Silver, \emph{Some applications of model theory in set theory}, Annals of
  Mathematical Logic \textbf{3} (1971), no.~1, 45--110.

\bibitem{Steel/Trang}
John Steel and Nam Trang, \emph{{$AD^+$}, derived models, and
  {$\Sigma_1$}-reflection}, 14 pp.

\bibitem{Szymanski/Xua83}
Andrzej Szyma{\'{n}}ski and Zhou~Hao Xua, \emph{The behaviour of
  ${\omega^{2^*}}$ under some consequences of {M}artin's axiom}, General
  Topology and Its Relations to Modern Analysis and Algebra, V, Heldermann,
  Berlin, 1983, pp.~577--584.

\bibitem{TodorcevicBK10}
Stevo Todorcevic, \emph{Introduction to {R}amsey {S}paces}, Princeton
  University Press, 2010.

\bibitem{TrujilloThesis}
Timothy Trujillo, \emph{Topological {R}amsey spaces, associated ultrafilters,
  and their applications to the {T}ukey theory of ultrafilters and {D}edekind
  cuts of nonstandard arithmetic}, Ph.D. thesis, University of Denver, 2014.

\bibitem{Woodin10}
W.~Hugh Woodin, \emph{Suitable extender models}, Journal of Mathematical Logic
  (2010).

\bibitem{Zheng17}
Yuan~Yuan Zheng, \emph{Selective ultrafilters on {FIN}}, Proceedings of the
  American Mathematical Society \textbf{145} (2017), no.~12, 5071--5086.

\bibitem{Zheng20}
\bysame, \emph{Parametrizing topological {R}amsey spaces},  (2018), 27 pp,
  Submitted.

\bibitem{Zheng18}
\bysame, \emph{Preserved under {S}acks forcing again?}, Acta Mathematica
  Hungarica \textbf{154} (2018), no.~1, 1--28.

\end{thebibliography}

\end{document}